\documentclass[10pt,reqno]{amsart}
\usepackage{amscd,amsmath,amssymb,amsfonts,verbatim}
\usepackage{times}
\usepackage[cmtip, all]{xy}
\usepackage{graphicx}
\usepackage[active]{srcltx}
\usepackage{color}
\usepackage{textcomp}
\usepackage[dvipsnames]{xcolor}
\usepackage[margin=1.25in]{geometry}
\definecolor{dpg}{RGB}{105,186,72}

\definecolor{refkey}{gray}{.5}   
\definecolor{labelkey}{gray}{.5} 
\definecolor{Red}{rgb}{1,0,0}

\newcommand{\ol}{\overline}		\newcommand{\wt}{\widetilde}

\newcommand{\ul}{\underline}

\newtheorem{theorem}{Theorem}[section]
\newtheorem{proposition}[theorem]{Proposition}
\newtheorem{lemma}[theorem]{Lemma}
\newtheorem{corollary}[theorem]{Corollary}
\newtheorem{question}[theorem]{Question}

\theoremstyle{definition}

\newtheorem{remark}[theorem]{Remark}
\newtheorem{example}[theorem]{Example}
\newtheorem{definition}[theorem]{Definition}

\newcommand{\ga}{\alpha}

	\newcommand{\gl}{\lambda}

\newcommand{\Pos}{\mbox{\rm Pos}}
\newcommand{\QM}{\mbox{\rm QM}}
\newcommand{\PO}{\mbox{\rm PO}}

	\newcommand{\BF}{\mbox{$\mathbb F$}}

	\newcommand{\BN}{\mbox{$\mathbb N$}}
	
\newcommand{\BQ}{\mbox{$\mathbb Q$}}	\newcommand{\BR}{\mbox{$\mathbb R$}}

	\newcommand{\CN}{\mbox{$\mathcal N$}}
	\newcommand{\CP}{\mbox{$\mathcal P$}}
\newcommand{\CQ}{\mbox{$\mathcal Q$}}	
\newcommand{\CS}{\mbox{$\mathcal S$}}

	\newcommand{\MP}{\mbox{$\mathfrak P$}}

\title{Descent problem for certificate of non-negativity on semi-algebraic sets}
\author{Manoj K. Keshari, Debapriya Ojha and Niladri Sekhar Patra}
\newcommand{\Addresses}{{
  \bigskip
  \footnotesize

\textsc{Manoj K. Keshari, Department of Mathematics, Indian Institute of Technology Bombay,
            Mumbai 400076, INDIA}\par\nopagebreak
  \textit{E-mail:} \texttt{<keshari@math.iitb.ac.in>}

\medskip
  
 \textsc{Debapriya Ojha, Department of Mathematics, Indian Institute of Technology Bombay,            Mumbai 400076, INDIA}\par\nopagebreak
 \textit{E-mail:}  \texttt{<debapriya.ojha1908@gmail.com>, <22D0788@iitb.ac.in>}

 \medskip
 
 \textsc{Niladri Sekhar Patra, Department of Mathematics, Indian Institute of Technology Bombay,          Mumbai 400076, INDIA}\par\nopagebreak
  \textit{E-mail:}  \texttt{<nilubanshichak@gmail.com>, <23D0790@iitb.ac.in>}

\medskip
  }}

\begin{document}

\begin{abstract}
Let $R$ be a real closed field and {$F$ a dense subfield of $R$ with the order topology on $R$.} 
Let {$\CS\subset R$ be a closed semi-algebraic set}  with $\partial \CS\subset F$. 
{Let $\CN=\{g_1,\ldots,g_s\}\subset F[x]$} be the natural choice of generators of $\CS$. 
{Let $f\in \Pos(\CS)$, i.e. $f\in F[x]$ and $f\geq 0$ on $\CS$. } 
Then we show that $f$ can be represented as    $$f=\sum_{e\in\{0,1\}^s } \sigma_e \,\ul g^e,$$
         where $\sigma_e\in \PO(1)= \sum F_{\geq 0} F[x]^2$ and  $\ul g^e=g_1^{e_1} 
 \cdots g_s^{e_s}$ (see \ref{1 thm}). 
 In other words, the preordering $\PO(\CN)$ of {$F[x]$, generated by $F_{\geq 0}\cup \CN$,}  is saturated (i.e. $\PO(\CN)=\Pos(\CS)$). 
 {When $F=R=\BR$,} this result is due to Kuhlmann and Marshall \cite[Thm 2.2]{Sal1}. 
 {Their proof also works when $F=R$ \cite[Prop 6.1.3]{Sch24}.} {When $F=\BF$} is an archimedean field  and $\CS \subset \BR$ is a compact {closed semi-algebraic set, then we show that}
       $\QM(\CN)=\PO(\CN)=\Pos(\CS)$ (see \ref{M=T}) (i.e. the quadratic module $\QM(\CN)$ of {$\BF[x]$, generated by $\BF_{\geq 0}\cup \CN$,} is saturated).
\end{abstract}
\vskip0.50in

\maketitle 
\subjclass 2020 Mathematics Subject Classification:{14P05, 14P10, 11E25}

 ~ {Sum of squares, basic closed semi-algebraic set, preordering, quadratic module}

\section{Introduction}

{ %
\noindent{\bf Notation 1.0.} Throughout the paper, we will use the following notations.
\begin{itemize}
    \item $R$ will denote a real closed field with order topology (\cite[p 14]{Sch24} or \cite[Section 2.6]{NS}) and $F$ will denote a dense subfield of $R$. 

\item $\BF$ will denote an archimedean field. Equivalently, $\BF$ is a subfield of 
the field of real numbers $\BR$ (see \ref{arch.0}).

\item $F[x]$ and $F[x_1,\ldots,x_n]$ will denote the polynomial ring over $F$ in one variable $x$ and $n$ variables $x_1,\ldots,x_n$ resp. 

\item {$g_0=1$}, $ \ul g=\{g_1,\ldots,g_s\}\subset A$ and $\ul g^e=g_1^{e_1}\ldots g_s^{e_s}$ for $e\in \{0,1\}^s$.

\item $\CS$ will denote a closed subset {of $R$ or $R^n$}.
\end{itemize}
 }

\begin{definition} 
Let $A$ denote $F[x]$ or $F[x_1,\ldots,x_n]$ and {$A_R$ denote} $R[x]$ or $R[x_1,\ldots,x_n]$ resp. 

\begin{enumerate}
    \item Let $\PO(1)$ and $\PO(1)_R$ denote the {\bf preorderings} of $A$ and $A_R$ resp. generated by $F_{\geq 0}$ and $\{1\}$ resp., where 
{ \[\PO(1)= \sum F_{\geq 0}A^2 =\left \{\sum_{i=1}^m a_i f_i^2 : a_i\in F_{\geq 0}, f_i\in A, m\geq 1\right\}  \]}
and 
{ \[ \PO(1)_{R}= \sum A_{R}^{2} =\left \{\sum_{i=1}^m f_i^2 : f_i\in A_R, m\geq 1\right\}. \] } 

\item Let $\PO(\ul g)$ and $\PO(\ul g)_R$ denote the {\bf preorderings} of $A$ and $A_R$ resp. generated by $F_{\geq 0}\cup \ul g$ and $\ul g$ resp., where
$$\PO(\ul g) = \sum_{e\in\{0,1\}^s} \PO(1)\, \ul g^e = \left \{\sum_{e\in\{0,1\}^s} \sigma_e \,\ul g^e : \sigma_e\in \PO(1) \right \}$$ and
$$\PO(\ul g)_R = \sum_{e\in\{0,1\}^s} \PO(1)_R\, \ul g^e = \left \{\sum_{e\in\{0,1\}^s} \sigma_e \,\ul g^e : \sigma_e\in \PO(1)_R \right \}.$$

\item 
     {Let $\QM(\ul g)$ and $\QM(\ul g)_R$ denote the {\bf quadratic modules} of $A$ and $A_R$ resp. generated by $F_{\geq 0}\cup \ul g$ and $\ul g$ resp., where 
     $$\QM(\ul g)= \sum_{i=0}^s \PO(1)\, g_i = \left \{\sum_{i=0}^s\sigma_i \, g_i : \sigma_i\in \PO(1) \right\}$$   and 
     $$\QM(\ul g)_{R}= \sum_{i=0}^s  \PO(1)_{R}\, g_i= \left \{\sum_{i=0}^s\sigma_i \, g_i : \sigma_i\in \PO(1)_{R} \right\}.$$ 
     } 

\item\label{4def}
{Let $\CS(\ul g)$ denote the {\bf basic closed semi-algebraic} set in $R$ or $R^n$ resp., where
$$\CS(\ul g)= \{x\in R \mbox{ or } R^n \mbox{ resp. }: g_i(x)\geq 0 \mbox{ for } i=1,\ldots,s\}.$$
} 
Note that $\CS(1)=R$ or $R^n$ resp.

\item {Let $ \Pos (\CS)$ and $\Pos (\CS)_R$ denote the set of all polynomials in $A$ and $A_R$ resp. which are non-negative on $\CS$, i.e.
$$\Pos (\CS)=\{f\in A:f\geq 0  \mbox{ on } \CS\} $$ 
and $$  \Pos(\CS)_{R}= \{f\in A_{R}:f\geq 0 \mbox{ on } \CS\}.$$}

\item Let {$\CS^\circ$ denote the interior of $\CS$ and $\partial \CS$ denote the boundary of $\CS$.}

\item Let $\CS\neq R$ be a finite union of closed intervals in $R$ with $\partial \CS\subset F$. 
Then $\CS=\CS({\mathcal N})$ is a basic closed semi-algebraic set in $R$ (see \ref{4def}), where $\mathcal{N}$ is the {\bf natural choice of generators} of $\CS$ consisting of the following elements of $F[x]$.

\begin{enumerate}
\item If $a\in \partial \CS$ and $(-\infty,a)\cap \CS=\emptyset$, then $x-a\in \mathcal{N}$.
\item If $a\in \partial \CS$ and $(a,\infty)\cap \CS=\emptyset$, then $a-x\in \mathcal{N}$.
\item If $a,b\in \partial \CS$, $a<b$ and $(a,b)\cap \CS=\emptyset$, then $(x-a)(x-b)\in \mathcal{N}$.
\end{enumerate}
For $\CS= R$, we take $\mathcal{N}=\{1\}$. Further, it is known that any closed semi-algebraic set $\CS$ in $R$ is a finite union of closed intervals (including points) \cite[Prop 2.1.7]{BCR}.

{
\item We will define properties $\MP_1$ to $\MP_5$ for later use. 
\begin{enumerate}
  \item  $\MP_1(F)$ is the property that the preordering $\PO(1)$ of $F[x]$ is {\bf saturated} (i.e. $\Pos(R)=\PO(1)$).
    
    \item $\MP_2(F,\CN)$ is the property that the preordering $\PO(\CN)$ of $F[x]$ is {\bf saturated} (i.e. $ \PO(\CN)=\Pos(\CS)$), 
    where $\CS \subset R$ is a closed semi-algebraic set with $\partial \CS \subset F $ and
    $\CN \subset F[x]$ is the natural choice of generators of $\CS$.
    Note that $\MP_2(F,\{1\})=\MP_1(F)$.
    
    \item Assume $\ul g\subset \BF[x_1,\ldots,x_n]$ such that $\CS(\ul g) \subset \BR^n$ is compact.
 Then $\MP_3(\BF,\ul g)$ is the {\bf Schm\"{u}dgen} property that $f>0$ on $\CS(\ul g) \implies f\in \PO(\ul g)$ for any $f\in \BF[x_1,\ldots,x_n]$. 
  
\item Assume $\ul g\subset \BF[x_1,\ldots,x_n]$ such that 
the quadratic module $\QM(\ul g)$ of $\BF[x_1,\ldots,x_n]$ is archimedean  (i.e.  
$N-\sum_1^n x_i^2\in \QM(\ul g)$
for some $N\in \BN$) (see \ref{Arch def}, \ref{Arch def1}). 
    Then $\MP_4(\BF,\ul g)$ is the {\bf Putinar} property that $f > 0$ on $\CS(\ul g) \implies f\in \QM(\ul g)$ for any  $f\in \BF[x_1,\ldots,x_n]$.

 \item Assume $\ul g \subset \BQ[x_1,\ldots,x_n]$ such that $\CS(\ul g)\subset \BR^n$ is compact. Then $\MP_5(\BF,\ul g)$ is the property that the quadratic module $\QM(\ul g)$ of $\BF[x_1,\ldots,x_n]$ is {\bf archimedean}.
\end{enumerate}
}
\end{enumerate}

\end{definition}

\subsection{} 
{Let $f\in R[x]$ such that $f \geq 0$ on $R$} (i.e. $f\in \Pos(R)_R$). 
Then it is well known that $f=f_1^2+f_2^2$ with $f_i\in R[x]$. 
Thus, $\PO(1)_{R}=\Pos(R)_{R}$ and {$\MP_1(R)$ is true. 
We can ask whether $\MP_1(F)$ is true?}

Landau \cite{L} proved that {$\MP_1(\BQ)$ is true}. 
In fact, he showed that any $f\in \BQ[x]\cap \Pos(\BR)$ is a sum of $8$ squares in $\BQ[x]$. 

Pourchet \cite{P} improved Landau's result by proving that any $f\in \BQ[x]\cap \Pos(\BR)$ is a sum of $5$ squares in $\BQ[x]$. 

Magron, El Din and Schweighofer \cite[Prop 12]{Sch1} gave an algorithmic proof to show that {$\MP_1(\BF)$ is true}. 
More precisely, they proved that any $f\in \BF[x] \cap \Pos(\BR)$ has a representation $f=\sum_1^{\deg(f)} a_i f_i^2$ with $a_i\in \BF_{\geq 0}$ and $f_i\in \BF[x]$, 
i.e. the number of weighted squares needed in the representation of $f$ is bounded by $\deg(f)$.

We observe that the proofs of Magron, El Din and Schweighofer \cite[Lem 10, Prop 11, Prop 12]{Sch1} work for {$F$ and $\MP_1(F)$ is true}. 
{Further, their proof gives better bound} that the number of weighted squares needed in the representation of $f$ is $1+ \frac 12\deg(f)$ (see \ref{sos}). {Note that we need to assume that $F$ is dense in $R$,} as this is needed in their proof  \cite{Sch1}. 
{Therefore, we can ask} the following 

\begin{question}\label{qn1.2}
    Let $R$ be a real closed field with order toplogy {and $F\subset R$ a subfield which is not dense in $R$}. {Is $\MP_1(F)$ true?}
\end{question} 

For an example of $F$ as in (\ref{qn1.2}), if  $R=\BR(t)_{\rm{alg}}^{\wedge}$ and $F= \mathbb{R}(t)$, then $F$ is not dense in $R$ (see \ref{rem.1}).

Note that if $F$ is not real closed, then $F_{\geq0}$ may not be equal to $\sum F^2$. Therefore, we have to allow weights from $F_{\geq 0}$. 
For an example, if $f=\sqrt 2$ is the constant polynomial in $\BQ(\sqrt 2)[x]$, then $f\in \Pos(\BR)$, but $f\notin \sum \BQ(\sqrt 2)[x]^2$. 
For another example, if $F=\BR(t)$, then $\PO(1) \neq \sum F[x]^2$ (see \ref{rem.1}). 
On the other hand, $\BQ_{\geq0}= \sum \BQ^2$ and $R_{\geq0}= R^2$.
\medskip

\subsection{} 
Kuhlmann and Marshall \cite[Thm 2.2]{Sal1} proved that {$\MP_2(\BR,\CN)$ is true} (i.e. $\Pos(\CS)_\mathbb{R}=\PO(\CN)_\mathbb{R}$). 
Their proof works over any $R$ \cite[Prop 6.1.3]{Sch24}, so $\MP_2(R,\CN)$ is true. 
We prove the following result (see \ref{1 thm}) which {shows that $\MP_2(F,\CN)$ is true.}
 
 \begin{theorem}\label{mt.1}
     Let { $\CS \subset R$ be a closed semi-algebraic set  with $\partial \CS \subset F $. Let
      $\mathcal N=\{g_1,\ldots,g_s\}\subset F[x]$ be the natural choice of generators of $\CS$.} Then 
     
\begin{enumerate}
  
       \item {$\MP_2(F,\CN)$ is true (i.e. $\Pos(\CS)=  \PO(\CN)$).}

         \item If $f\in  \Pos(\CS)$ and $\deg(f)=n$, then $f$ has a representation 
         $$f=\sum_{e\in\{0,1\}^s } \sigma_e \,\ul g^e,$$
         where $\sigma_e\in \PO(1)$ and
         $\deg(\sigma_e) = d_e \leq n-\deg(\ul g^e)$. 
         Further, $\sigma_e$ has a representation $\sigma_e=\sum_{i=1}^{1+\frac 12 d_e} a_{e,i}\,f_{e,i}^2 $ with $a_{e,i}\in F_{\geq 0}$ and $f_{e,i}\in F[x]$.
     \end{enumerate}
 \end{theorem}
 
\subsection{}
 {Let $\ul g \subset F[x]$ and $\CQ=\QM(\ul g)$ be the quadratic module in $F[x]$}. Let
  {$F[x]_d=\{f\in F[x] :\deg(f)\leq d\}$} and 
$ \CQ_d=\CQ\cap F[x]_d$. Let 
 $$ \CQ_{[d]}=\left\{\sum_0^s \sigma_i\, g_i : \sigma_i\in \PO(1) \mbox{ and } \deg(\sigma_i \,g_i)\leq d\right\}.$$ 
 {The quadratic module $\CQ$ is called {\bf stable}} if $\CQ_d\subset \CQ_{[\phi(d)]}$ for all $d$, 
 where the map $\phi:\BN\to \BN$ is called the {\bf stability bound} for $\CQ$ with respect to the system of generators $\ul g$ (see \cite[section 6.6]{Sch4}).
 As a consequence of (\ref{mt.1}(2)), taking $\CQ=\PO(\CN)$, we get that $\CQ_d=\CQ_{[d]}$.
In other words, the quadratic module $\PO(\CN)$ is stable and the stability bound for $\PO(\CN)$ with respect to the system of generators 
$\{\ul g^e:e\in \{0,1\}^s\}$ is the identity map $:\BN\rightarrow\BN$. 

\subsection{}
{F.-V. Kuhlmann \cite[Thm 1]{fvkhul} proved the following result: If $(L,v)$ is a henselian non-trivially valued field (with the valuation topology on $L$) and having residue field $Lv$ of characteristic $0$, then $(L,v)$ has a proper dense subfield. 
If $L=R$ is a non-archimedean real closed field, then the natural valuation of $R$ is non-trivial and henselian with a (real closed) residue field $Lv$ of characteristic $0$ \cite[page 1]{fvkhul}.
Further,
the valuation topology on $R$ coincided with the order topology. Therefore, we get that there exists a proper dense subfield $F$ of $R$ with order topology on $R$. }
Since $\BQ$ is proper dense subfield of $\BR$, we get that any $R$ has a proper dense subfield.
Therefore, (\ref{mt.1}) is applicable in all cases.
\medskip

\subsection{}  
Kuhlmann and Marshall \cite{Sal1} and Kuhlmann, Marshall and Schwartz \cite{Sal2} obtained many consequences of their  result \cite[Thm 2.2]{Sal1}. 
As an application of (\ref{mt.1}), we obtain the following result (see \ref{Non cpt sat}).

 \begin{theorem}\label{mt.2}
Assume $\ul g \subset F[x]$ and
{$\CS = \CS(\ul g) \subset R$ is an unbounded closed semi-algebraic set } with $\partial \CS\subset F$. 
 Then 

 \begin{enumerate}
     
     \item the preordering $\PO(\ul g)$ of $F[x]$ is saturated 
     (i.e. {$\Pos(\CS)=\PO(\ul g)$) if and only if} $\ul g$ contains all the natural choice of generators of $\CS$, upto scaling by $F_{> 0}$.
     
     \item Let $\mathcal{N}\subset F[x]$ be the natural choice of generators of $\CS$. Then the quadratic module $\QM(\CN)$ of $F[x]$ is saturated (i.e. {$\Pos(\CS)=\QM(\CN)$) if and only if} either $|\CN|\leq 1$ or $|\CN|=2$ and $\CS$ has an isolated point. 
 \end{enumerate}
 \end{theorem}  
 
\subsection{} 
Sturmfels asked the following descent problem  \cite[Qn 1.1]{H}:
Let {$f\in \BQ[x_1,\ldots,x_n] \cap \sum \BR[x_1,\ldots,x_n]^2$. Is $f\in \sum \BQ[x_1,\ldots,x_n]^2$? }

If $f\in \BQ[x_1,\ldots,x_n] \cap \sum \BR[x_1,\ldots,x_n]^2$, then by Tarski transfer principle \cite{BCR}, $f\in R_0[x_1,\ldots,x_n]^2$, where $R_0= \BR \cap \ol \BQ$ is the real closure of $\BQ$. 
There exists a number field $L\subset R_0$ such that  $f\in \sum L[x_1,\ldots,x_n]^2$. Hillar \cite[Thm 1.4]{H} proved that if $L$ is totally real, then $f\in \sum \BQ[x_1,\ldots,x_n]^2$. 

Scheiderer \cite[Prop 1.6]{S} generalized Hillar's result as follows: Let $E/k$ be a finite totally real extension of real fields
(i.e. every ordering of $k$ extends to $[E:k]$ different orderings of $E$).  Let $\ul g\subset k[x_1,\ldots,x_n]$ and $f\in k[x_1,\ldots,x_n]$ such that $f\in \sum_0^s E[x_1,\ldots,x_n]^2\,g_i$ (note $g_0=1$). 
{Then Scheiderer} \cite[Prop 1.6]{S}  proved that $f\in \sum_0^s k[x_1,\ldots,x_n]^2\, g_i$.
Note that by taking $\ul g$ as $\{\ul g^e : e\in \{0,1\}^s\}$, this covers the preordering case.

Scheiderer \cite[Thm 2.1]{S} also gave examples showing that the {question of} Sturmfels has a negative answer in general.

\subsection{}  
{Schm\"{u}dgen \cite[Cor 3]{Kon} proved that $\MP_3(\BR,\ul g)$ is true. W\"{o}rmann \cite{Wor} generalized it and proved that $\MP_3(\BF,\ul g)$ is true
(see \ref{Pow thm1}). 
}

{Putinar \cite{Putinar} proved that
$\MP_4(\BR,\ul g)$ is true.
Powers \cite[Thm 7]{Vict2} proved that $\MP_4(\BQ,\ul g)$ is true. We observe that Powers proof works over arbitrary $\BF$ and $\MP_4(\BF,\ul g)$ is true (see \ref{Pow thm2}). Note that for non-archimedean $R$, both $\MP_3(R,\ul g)$ and $\MP_4(R,\ul g)$ are false  \cite[Ex 5.5.5]{Sch24}. }

\subsection{} 
Next we consider the analogue of (\ref{1 thm}) when $\partial \CS \not\subset F$.

\begin{theorem} 
Assume $\ul g \subset F[x]$ and 
 $\CS= \CS(\ul g) \subset R$ is a closed semi-algebraic set. {Assume $\partial \CS  \not\subset F$}. Then we prove the following results.

\begin{enumerate}
    \item Assume $\CS$ is unbounded and there {exists} $c\in \partial \CS $ with $[F(c):F]\geq 3$. Then $\Pos(\CS)$ is not a finitely generated preordering of $F[x]$ (see \ref{t4.1}).
    
    \item Assume $\CS=[c,c']$ with $c,c'\in R$ conjugates over $F$ and $[F(c):F]=2$. Then $\Pos(\CS)= \PO(g)$, where $g=-(x-c)(x-c')$ (see \ref{t4.5}).
    
    \item Assume $[F(c):F]\leq 2$ for all $c\in \partial \CS$. If $[F(c):F]=2$ for some $c\in \partial \CS$ and $c<c'$ are conjugates over $F$, then further assume that
$c'\in \partial \CS$,
$(c,c')\cap \CS=\emptyset$ and $(x-c)(x-c')\in \ul g$. 
   Then $\Pos(\CS)=\PO(\ul g)$ (see \ref{t4.5}).

    \item Let $d\in F_{>0}$ and $\sqrt d \not\in F$. Then 
    \begin{enumerate}
        \item $\Pos([0,\sqrt d])=\PO(x,d-x^2)$ in $F[x]$ (see \ref{t4.31}).
        
        \item $\Pos([\sqrt d,\infty))$ is not a finitely generated preordering of $F[x]$ (see \ref{t4.3}).
    \end{enumerate}
     
\end{enumerate}
\end{theorem} 

We pose some questions about finite generation of $\Pos(\CS)$ in case $\partial \CS \not\subset F$ (see \ref{Qn}).

\subsection{} 
{Powers \cite[Remark 8]{Vict2} asked whether $\MP_5(\BR,\ul g)$ is true $\implies \MP_5(\BQ,\ul g)$ is true? }  
  We show that the answer is affirmative, when {the} dimension of the ring $\BQ[x_1,\ldots,x_n]/(\QM(\ul g) \cap - \QM(\ul g))$ is zero (i.e. the support $\QM(\ul g)\cap - \QM(\ul g)$ of the quadratic module $\QM(\ul g)$ has codimension zero) (see \ref{quad in rp}). {In this case, $\CS(\ul g) \subset \BR^n$ is a finite set.} This result is due to Baldi, Krick and Mourrain \cite{Ba}. We give a different proof which is inspired from W\"{o}rmann's trick \cite{Wor}.

  \subsection{} 
  {By (\ref{1 thm}), $\Pos(\CS)=\PO(\CN)$ is a  finitely generated preordering, where $\CS\subset \BR$ is a closed semi-algebraic set with $\partial \CS\subset \BF $ and $\CN\subset F[x]$ is the natural choice of generators of $\CS$. It is natural to ask, if we can give a similar description of $\Pos(\CS)$, when $\CS \subset \BR$ is a closed but non-semi-algebraic set.} We prove the following result (see \ref{t.2}).

  \begin{theorem} 
  Let $C=\cap_{n\geq 0} C_n$ be the Cantor set in $\BR$, where $C_{0}=[0,1]$, $C_1=[0,\frac 13]\cup [\frac 23,1]$, $C_2=[1,\frac 19]\cup [\frac 29, \frac 39]\cup [\frac 69, \frac 79]\cup[\frac 89, 1], $ etc. Let  $\CN_n \subset \BF[x]$ be the natural choice of generators of the closed semi-algebraic set $C_n$. Then we show that
    ${\displaystyle  \Pos(C)= \bigcup_{n\in\mathbb{N}} \Pos(C_n)=\bigcup_{n\in\mathbb{N}} \PO({\mathcal{N}_n})}.$
    \end{theorem} 

\subsection{} 
   We end this section with a few words on the method {of proof of (\ref{1 thm}).} 
  In the case $F=R=\BR$, the proof of Kuhlmann and Marshall \cite[Thm 2.2]{Sal1} proceeds by induction on the degree of the polynomial $f\in \Pos(\CS)$.
 {Assume $f> 0$ on $\CS$ and there exist $ a,b\in \partial \CS$ with $(a,b) \cap \CS= \emptyset$ such that $f(\ga)<0$ for some $\ga \in (a,b)$. Then there exist $c<d $ in $(a,b)$ such that $f(c)=0=f(d)$ and $f>0$ on $[a,c) \cup (d,b]$. Then $f=(x-c)(x-d)g $ with $g\in \Pos(\CS)$. The result holds for $g$, by induction on degree. Hence the result holds for $f$, using Berg-Maserick (see \ref{CM}). }
The main difficulty, in the case $F \neq R$, is that $c, d$ may not lie in $F$. {So,} we can not use $f=(x-c)(x-d)g$, as in the case $F=R$. Therefore, we need to do case by case analysis, depending on
{
\begin{enumerate}
    \item the geometry of the semi-algebraic set $\CS$ (i.e. whether $\CS$ is connected, {not connected}, bounded or unbounded). 
    
 \item the position of the zeros of $f$ in $\CS$ (i.e. whether the zeros of $f$ lie in $\CS^\circ$ or $\partial\CS$).
 
 \item whether the minimum of $f$ over $\CS$ occurs at a point in $\CS^\circ$ or $\partial \CS$.
    \end{enumerate} 
    }
   Overall, our proof is substantially more involved than that of Kuhlmann and Marshall.


\section{Preliminaries}

Recall that $R$ will denote a real closed field, $F\subset R$ a dense subfield of $R$ and $\BF\subset \BR$ will denote an archimedean field.

\begin{definition}
An ordered field $(K,\leq)$ is called {\bf archimedean} if for all $a,b\in K$ with $b>0$, there exist $n\in \BN$ such that $nb>a$. 
\end{definition}

The following result is due to H\"{o}lder \cite[Thm 1.1.18]{Sch4}.

\begin{theorem}\label{arch.0}
    Let $(K,\leq)$ be an archimedean ordered field. Then there exists a unique order preserving homomorphism $\phi:K \to \BR$. In other words, $K$ has a unique order preserving embedding into $\BR$. Thus, $\BQ\subset K \subset \BR$ is equivalent to $K$ being archimedean.
\end{theorem}

\begin{definition}\label{Arch def}
Let $A$ be a commutative ring with $\mathbb{Q}\subseteq A$. Then $M\subset A$ is a {\bf quadratic module} of $A$ if $M+M\subset M$, $A^2M\subset M$, $1\in M$ and $-1\not\in M$. The quadratic module $M$ of $A$  is called {\bf archimedean} if for any $a\in A$,  there exists  $n\in \BN $ (depending on $a$) such that $n+a \in M$.
    \end{definition}

For the proof of the following result, see \cite[Cor 5.2.4]{Mur1} and \cite[Lem 4.3.4]{Sch4}.

\begin{lemma}\label{Arch def1}
Let {$M$} be a quadratic module  of   $\BF[x_1,\ldots,x_n]$. Then $M$ is archimedean if and only if $N- \sum_{i=1}^{n} {x_i}^{2}\in M $ for some  $N\in \BN$.   
\end{lemma}

The following result is due to Magron,  El Din and Schweighofer \cite[Prop 12]{Sch1} in the case $R=\mathbb{R}$. They proved that the number of weighted {squares, needed in the representation of $f\in \Pos(R)$,} is $\leq \deg(f)$. We improve it to $1+\frac 12 \deg(f)$. 

\begin{theorem}\label{sos}
Let $f\in F[x]$. If $f\in \Pos(R)$, then $f$ has a representation $$f=\sum_1^{1+\frac 12\deg(f)} a_i f_i^2$$ with $a_i\in F_{\geq 0}$ and $f_i \in F[x]$. In particular,
$\Pos(R)=\PO(1)$. 
\end{theorem}

\begin{proof}
    {The proofs of \cite[Lem 9, 10, Prop 11, 12]{Sch1} are given in the case $F= \BF\subset \mathbb{R}$. We note that the same proofs work for $F\subset R$. 
    In the proof of \cite[Prop 12]{Sch1}, we need to choose $t \in F \cap (a-\epsilon,a)$ (using \cite[Prop 11]{Sch1}), where $a\in R$} is the smallest global minimizer of $f$. Therefore, we need to assume that $F$ is  dense in $R$. This will prove the first part that $\Pos(R)=\PO(1)$. 
    
    We will show that if $f\in \Pos(R)$ and $\deg(f)=n$ (which is even), then there exist a representation $f=\sum_1^{1+\frac n2} a_i f_i^2$ with $a_i\in F_{\geq 0}$ and $f_i\in F[x]$ (clearly, $\deg(f_i)\leq \frac n2$). We will use induction on $n=\deg(f)$.

 If $n=2$, then by \cite[Lem 9]{Sch1}, $f=af_1^2+b$ with $f_1\in F[x]$ and $a,b\in F_{\geq 0}$. {Thus,} we are done in this case. {Assume  $n \geq 4.$}
    
    {We first assume $f(a)=0$} for some $a\in R$. Then $a$ is a zero of $f$ of even multiplicity. Further, $a$ is algebraic over $F$. Let $g\in F[x]$ be the minimal polynomial of $a$ over $F$. Then $f=g^{2}h$ with $h\in \Pos(R)$  and $\deg(h)= m\leq n-2$. By induction hypothesis, $h =\sum_1^{1+\frac m2} b_i h_i^2$ with $b_i \in F_{\geq 0}$ and $h_i \in F[x]$. Thus, $f= \sum_1^{1+\frac m2} b_i (gh_i)^2$ and we are done.
    
   {Now we assume  $f>0$ on $R$.}
     By \cite[Prop 11, 12]{Sch1}, there exist $t\in F$ and $f_t\in \Pos(R)$ such that 
     \begin{enumerate} 
     \item $\deg(f_t)=2$ and $f_t = a l^2$ with $a\in F_{\geq 0}$ and $l\in F[x]$, 
     \item $f-f_t$ vanishes at $t\in F$ and 
     \item $f-f_t\in \Pos(R)$.
    \end{enumerate}
    Let $f-f_t = (x-t)^{2} g(t)$, where $g\in \Pos(R)$ and $\deg(g)=n-2$. By induction hypothesis, $g=\sum_1^{1+ \frac {n-2}2} c_i g_i^2$ with $c_i\in F_{\geq 0}$ and $g_i \in F[x]$. Then $f= al^2+ \sum_1^{1+ \frac {n-2}2} c_i (x-t)^2g_i^2$. This completes the proof.
    $\hfill\square$
\end{proof}

\begin{remark}\label{rem.1} Let $F=\BR(t)$ be the ordered field {with ordering} $P_{0,+}$ (i.e. $0<t<r$ for all $r\in \BR_{>0}$) \cite[Example 1.1.2]{BCR}. The real closure of $F$ is the field $R=\BR(t)_{\rm{alg}}^{\wedge}$ of real  algebraic Puiseux series \cite[Example 1.3.6]{BCR}.

\begin{enumerate} 

\item  Let $f(x) = (x^2-t)^2-t^3 \in F[x]$. Then $f\in  \Pos(F)$, but $f\not\in \Pos(R)$ \cite[Example 6.1.7]{BCR}. In particular, $f\not\in \sum R^2$. Further, it shows that $F$ is not dense in $R$.

\item 
 Let $f(x) = (x^2 + t)^2 - t^3 \in F[x]$. Then $f=x^4 + 2tx^2 + (t^2-t^3) \in {\PO(1)}$ (since $t^2-t^3>0$ in $F$).
 Since $F$ is not dense in $R$, we can not conclude that $f\in \Pos(R)$. To see this, note that
 $f=(x^2+t+t^{3/2})(x^2+t-t^{3/2})$ in  $R[x]$.  Since $t+t^{3/2},t-t^{3/2} \in R_{>0}$, we get $f>0$ on $R$ and hence $f\in \sum R[x]^2$.
 
 \noindent{\it Claim:} $f \not\in  \sum F[x]^2$ (hence $\Pos(R)\neq \sum F[x]^2$).
 
 If $f\in \sum F[x]^2$, then $f(0)\in \sum F^2$. Hence $f(0) \geq 0$ in all the orderings of $F$. Consider the ordering $P_{\infty, +}$ of $\BR(t)$ (i.e. $0<t^{-1} <r$ for all $r\in \BR_{>0}$). In this ordering, $f(0) = t^2-t^3 =t^3(t^{-1}-1)<0$, a contradiction. This proves the claim.

\item {{\it Question :} Is $\Pos(R) = \PO(1)?$  Note that $F$ is not dense in $R$, so (\ref{1 thm}) does not apply.}
$\hfill \square$

 \end{enumerate}
 \end{remark}

The following result is due to Berg and Maserick \cite[Lem 4]{B-M} in the case $F=\BR$ (see also \cite[Lem 2.7.4]{Mur1}). 
The same proof works for any subfield of $R$. We reproduce the proof below. See (\ref{CM2}) for a generalization.
 
\begin{lemma}\label{CM}
 Let $F$ be any subfield of $R$ ($F$ need not be dense in $R$). Let $a\leq c<d \leq b$ be elements of {$F.$ Let}  $g= (x-a)(x-b)$ and $f= (x-c)(x-d)$. {Then $f =\gl g + a_1l_1^2+a_2\in T_{\{g\}}$, where $\gl,a_1,a_2\in F_{\geq 0}$ and $l_1 \in F[x]$ of degree $\leq 1$. } In other words, $f=\sigma_0+\sigma_1g$, where $\sigma_0=a_1l_1^2+a_2$  and $\sigma_1=\gl$ with $\sigma_0,\sigma_1 \in \PO(1)$ .
\end{lemma}

\begin{proof}

 It is enough to show that there {exists} $\gl\in F_{\geq 0}$ such that $\phi(x) =  f-\gl g\geq 0$ on $R$. {Then by
 \cite[Lem 9]{Sch1}, $\phi = a_1l_1^2+a_2$ with $a_1,a_2\in F_{\geq 0}$ and $\deg(l_1)\leq 1$. Thus, $f=\gl g + a_1l_1^2+a_2$. }
 
    By change of variable $x\mapsto \frac{2}{b-a}(x-a)-1$, we may { assume  $a=-1$ and $b=1$. Thus,} $-1\leq c< d\leq 1$. 
\begin{enumerate}
    \item {When} $c+d=0$, take $\gl=1$. Then $\phi (x)=cd+1 \in F_{\geq 0}$.
    
    \item {When} $\frac{c+d}{2}\in (0,1)$, take $\gl=1-\frac{c+d}{2} \in F_{\geq 0}$. Then the minimum value of $\phi$ is $$cd+\gl-\frac{1}{1-\gl}(\frac{c+d}{2})^{2}=(1-c)(1-d)\in F_{\geq 0}.$$ 

\item {When} $\frac{c+d}{2}\in (-1,0)$, take $\gl=1+\frac{c+d}{2} \in F_{\geq 0}$. Then the minimum value of $\phi$ is $$cd+\gl-\frac{1}{\gl-1}(\frac{c+d}{2})^{2}=(1+c)(1+d)\in F_{\geq 0}$$
    
\end{enumerate}
This completes the proof. $\hfill \square$ 
\end{proof}

The following result is due to Schm\"{u}dgen \cite[Cor 3]{Kon} in the case $\BF=\BR$. 
The general case is due to   W\"{o}rmann \cite{Wor}, as noted by Schweighofer \cite[Thm 1.2]{Sch3}. 

\begin{theorem}\label{Pow thm1}
Let $\ul g\subset \mathbb{F}[x_1,\ldots,x_n]$ such that the basic closed semi-algebraic set $\CS(\ul g)\subset \mathbb{R}^n$ is compact. If $f\in \mathbb{F}[x_1,\ldots,x_n]$ and $f>0$ on $\CS(\ul g)$, then
$f\in \PO(\ul g)$. In particular, $\PO(\ul g)$ is an archimedean preordering of $\mathbb{F}[x_1, \ldots, x_n]$.
\end{theorem}

 In {the} case $\BF= \BR$, the following result is due to Putinar \cite{Putinar} and in  the case $\BF=\BQ$, it is due to Powers \cite[Thm 7]{Vict2}. {We observe that Powers proofs work for arbitrary $\BF$.} 
 
\begin{theorem}\label{Pow thm2}
Let $\ul g\subset \BF[x_1,\ldots,x_n]$ such that
$ N-\sum_{1}^n x_i^2\in \QM(\ul g)$ for some $N\in \mathbb{N}$ (i.e. $\QM(\ul g)$ is an archimedean quadratic module of $\BF[x_1,\ldots,x_n]$).
If $f\in \mathbb{F}[x_1,\ldots,x_n]$ and $f>0$ on $\CS(\ul g)$, then  $f\in \QM(\ul g)$.
\end{theorem}

{Using (\ref{Pow thm1}, \ref{Pow thm2})},  we extend some results from Kuhlmann, Marshall and Schwartz \cite{Sal2}, where {they are proved} for $\BF=\BR$. We will follow the proofs from \cite{Sal2} with suitable modifications.

\begin{corollary}\label{1st cor}
Let  $\ul g\subset \BF[x_1,\ldots,x_n]$ such that  $\CS(\ul g)\subset \BR^n$ is compact. Let $f, g\in \BF[x_1,\ldots,x_n]$ be $ \geq 0$ on $\CS(\ul g)$. Assume $f$ and $g$ are relatively prime modulo the ideal $\PO(\ul g) \cap -\PO(\ul g)$ {and $fg \in \PO(\ul g)$.} Then $f,g \in \PO(\ul g)$.
\end{corollary}

\begin{proof}
Follow the proof of \cite[Cor 2.2]{Sal2} with
$C = \CS(\ul g)$, $A = \BF[x_1,\ldots, x_n]/(\PO(\ul g) \cap -\PO(\ul g))$ and use (\ref{Pow thm1}).
$\hfill \square$
\end{proof}

\begin{corollary}\label{prime}
 {Let $f, g \in \BF[x_1,\ldots,x_n]$ be relatively prime. Assume $\CS( fg)\subset\BR^n$ is compact and $f, g \geq 0$ on $\CS(fg)$. Then }
\begin{enumerate}
    \item (Basic lemma) there exist $\sigma, \tau \in \sum \BF_{\geq 0} \BF[x_1,\ldots,x_n]^2$ such that $1 = \sigma f + \tau g$.
    \item $\QM(fg) = \QM(f, g) = \PO(fg) = \PO(f, g)$.
\end{enumerate}
\end{corollary}

\begin{proof}
Follow the proof of \cite[Cor 2.3]{Sal2} with $C=\CS(fg)$, $A=\BF[x_1,\ldots,x_n]$ and use (\ref{Pow thm1}).
$\hfill \square$
\end{proof}

\begin{corollary}\label{cor.1}
Let  $\ul g\subset \BF[x_1,\ldots,x_n]$  such that  $\CS(\ul g)\subset\BR^n$ is compact. {Let $f\in \BF[x_1,\ldots,x_n]$ be $ \geq 0$ on $\CS(\ul g)$ and $\CP=\PO(\ul g)$.
Assume $f \in \CP + f \sqrt{(f) +  (\CP \cap -\CP)}
$.
Then $f \in \CP$.}
\end{corollary}

\begin{proof}
Follow the proof of \cite[Cor 2.4]{Sal2} and use (\ref{1st cor}).
$\hfill \square$
\end{proof}

\begin{corollary}
Let $\ul g\subset \BF[x_1,\ldots,x_n]$ such that  $\CS(\ul g) \subset \BR^n$ is compact. {Let  $f\in \BF[x_1,\ldots,x_n]$ be $ \geq 0$ on $\CS(\ul g)$. Assume $f = \sigma + \tau b$, where $\sigma, \tau \in \PO(\ul g)$ and $b\in \BF[x_1,\ldots,x_n]$. Further assume  $f = 0$ on $\CS(\ul g)\Rightarrow b > 0$ on $\CS(\ul g)$. } Then $f \in \PO(\ul g)$.
\end{corollary}

\begin{proof}
Follow the proof of \cite[Cor 2.5]{Sal2} and {use (\ref{Pow thm1}, \ref{cor.1}.)}
$\hfill \square$
\end{proof}


\section{Main Theorem : Saturation of preordering $\PO(\CN)$ in $F[x]$}

In this section, we will prove our main result. (1) is due to Kuhlmann and Marshall, see \cite[Thm 2.2]{Sal1} when $F=\BR$ and \cite[Prop 6.1.3]{Sch24} when $F=R$.

\begin{theorem}\label{1 thm}
Let {$\CS\subset R$ be a closed semi-algebraic set with $\partial \CS\subset F$. Let $\mathcal{N}= \{g_1,\ldots,g_s\} \subset F[x]$} be the natural choice of generators of $\CS$. (Note that $\CS$ is a 
finite union of closed intervals in $R$).  Then   
\begin{enumerate}
         \item $\Pos(\CS)=\PO(\CN)$ (i.e. the preordering $\PO(\CN)$ of $F[x]$ is saturated).
         
         \item If $f\in \Pos(\CS)$ and $\deg(f)=n$, then $f$ has a representation
         $$f=\sum_{e\in\{0,1\}^s } \sigma_e \,\ul g^e,$$
         where $\sigma_e\in \PO(1)$ and $\deg(\sigma_e)= d_e\leq n-\deg(\ul g^e)$. {By (\ref{sos}),}  $\sigma_e$ has a representation $\sigma_e=\sum_{i=1}^{1+\frac 12 d_e} a_{e,i}\, f_{e,i}^2$ with $a_{e,i}\in F_{\geq 0}$ and $f_{e,i}\in F[x]$.
     \end{enumerate}
\end{theorem}

\begin{proof} We will prove the result in cases depending on the geometry of $\CS$ and the property of the polynomial $f\in \Pos(\CS)$. The main idea of the proof is as follows. 

First  we consider the case when $\CS$ is connected (see \ref{thm.1}). 
When $\CS$ has atleast 2 connected components, we first reduce to the case that there exist $a,b\in \partial \CS\subset F$ with $(a,b)\cap \CS=\emptyset$ and 
$f(x)<0$ for some $x\in (a,b)$ (see \ref{2.1}). Then we consider the case when the global minimum of $f$ over $\CS$ occurs at $\CS^\circ$ (see \ref{thm.2}). Finally, we consider the case when the global minimum of $f$ over $\CS$ occurs at $\partial \CS$ (see \ref{thm.3}). For the proof of the last case,  we first reduce to the case that $f>0$ on $\CS^\circ$ and $f(c)=0$ for some $c\in \partial \CS$. Finally, we show that there exist $g\in \PO(\CN)$ with $\deg(g)\leq \deg(f)$ and $f-g=hh'$ with $\deg(h')\geq 1$, $h'\in \PO(\CN)$ and $h\in \Pos(\CS)$. The proof of this part is done by considering the behavior of $f$ at all points of $\partial \CS$. By induction hypothesis on degree, we get $h\in \PO(\CN)$ and {hence,} $f=g+hh'\in \PO(\CN)$. 
In each step of the proof, we will indicate that (2) also holds. To prove the result, we will use induction on the degree of the polynomial in $\Pos(\CS)$.
By (\ref{sos}), we may assume that $\CS\neq R$.
{Let $f\in \Pos(\CS)$. If $\deg(f)=0$, then $f\in F_{\geq 0} \subset \PO(\mathcal{N})$ and (2) also holds. {Thus,} we assume that $\deg(f) =  n\geq 1$.}
 
\subsection{}\label{thm.1} 
Assume $\CS$ is connected.
\subsubsection{}\label{0 thm} First assume $\CS=[a,\infty)$ (resp. $(-\infty,a]$) for some $a\in F$ (i.e. $\CS$ is a connected and unbounded closed interval in $R$).

{Replacing $x$ by $x-a$ (resp. $a-x$), we may assume $\CS=[0,\infty)$. Then $\mathcal{N}=\{x\}$.}
Note that $f(x^2)\in \Pos({R})$ and $\deg(f(x^2))=2n$. By (\ref{sos}), 
$$f(x^2)=\sum_{i=1}^{1+n} t_i\, q_{i}^2$$ with $t_i\in F_{\geq 0},\, q_{i}\in F[x]$ and $\deg(q_i)\leq n$. 
Write $q_i(x)=\alpha_i(x^2)+x\,\beta_i(x^2)$ in $F[x]$.  Then $\deg(\alpha_i(x))\leq \frac{n}{2}$ and $\deg(\beta_i(x))\leq \frac{n-1}{2}$. Then
\[f(x^2)
= \sum_{i=1}^{n+1}t_i \,(\alpha_i(x^2)^2+x^2\,\beta_i(x^2)^2)+2x\,\sum_{i=1}^{n+1}t_i\,\alpha_i(x^2) \,\beta_i(x^2)
\]
Each monomial in $f(x^2)$ is of even degree, hence $\sum_{i=1}^{n+1} t_i\,\alpha_i(x^2)\, \beta_i(x^2)=0$ and $$f(x^2)=\sum_{i=1}^{n+1} t_i\,(\alpha_i(x^2)^2+x^2\,\beta_i(x^2)^2).$$
Hence 
$$f(x)=\sum_{i=1}^{n+1} t_i\,(\alpha_i(x)^2+x\,\beta_i(x)^2)=\sigma_0+x\,\sigma_1,$$ 
where 
$\sigma_0=\sum_{i=1}^{n+1} t_i\,\alpha_i(x)^2$ and $\sigma_1=\sum_{i=1}^{n+1}t_i\, \beta_i(x)^2\in \PO(1)$. Further 
$\deg(\sigma_0)\leq n$ and $\deg(\sigma_1)\leq n-1$. {Thus,} $f\in \PO(\mathcal{N})$ and (2) also holds. {Thus,} we are done in the case (\ref{0 thm}).

\subsubsection{} \label{0} Assume $\CS=\{a\}$ is a singleton with $a\in F$. 

Replacing $x$ by $x-a$, we may assume that {$\CS=\{0\}$ and $\mathcal{N}=\{x,-x\}$}. Let $f=\sum_{i=0}^n a_i\,x^i$. 
Then $a_0\geq 0$. Let us write 
$$ \sum_{\substack{i \text{ even} \\ a_i > 0}} a_i\, x^i=\sigma_0, \; \sum_{\substack{i \text{ odd} \\ a_i > 0}} a_i\, x^i = x \sigma_1,\;  \sum_{\substack{i \text{ odd} \\ a_i < 0}} a_i\, x^i=(-x)\sigma_2,\; \sum_{\substack{i \text{ even} \\ a_i < 0}} a_i \,x^i=(-x^2)\sigma_3.$$ Then $\sigma_i\in {\PO(1)}$. {Thus,} 
$$
f=\sigma_0+x \sigma_1+(-x)\sigma_2+(-x^2)\sigma_3\in \PO(\mathcal{N})
$$ and (2) is also satisfied. {Thus,} we are done in the case (\ref{0}).

\subsubsection{} \label{1} Assume $\CS=[a,b]$ with $a<b\in F$ (i.e. $\CS$ is a connected and bounded closed interval in $R$).

In this case, the result is classical when $F=R=\mathbb R$. The proof given in \cite[Cor 3]{Vict1} works in our case, using  (\ref{0 thm}). We will sketch the proof. By change of variable, we may assume that $\CS=[-1,1]$. If  $\deg(f)=n$, then the $n$-th degree Goursat transform of $f$ is defined by $$\wt f(x)=(1+x)^n f\left(\frac {1-x}{1+x}\right).$$ It has the property that $\wt f\in \Pos([0,\infty))$ and $\deg(\wt f)\leq n$. By (\ref{0 thm}), we can write $\wt f(x) = \sigma_0(x)+x\,\sigma_1(x)$, where $\sigma_0,\sigma_1\in {\PO(1)}$, $\deg(\sigma_0)\leq n$ and $\deg(\sigma_1)\leq n-1$. 
The Goursat transform of $\wt f$ is
$$\wt{\wt f}(x)=2^n f(x) \in \PO(x+1,1-x).$$ 
Thus 
$f\in \PO(\mathcal{N})$. Now 
$$2^n f(x)=\wt {\wt f}(x)=(1+x)^n \sigma_0\left(\frac{1-x}{1+x}\right)+(1+x)^n\left(\frac{1-x}{1+x}\right)\sigma_1\left(\frac{1-x}{1+x}\right),$$ where
\begin{align*}
(1+x)^n \sigma_0\left(\frac{1-x}{1+x}\right)=\begin{cases}
\wt{\sigma}_0 (x)&\text{ with }\deg (\wt{\sigma}_0) \le n , \quad \text{ if } n \text{ even}\\
(1+x) \,\wt{\sigma}_0(x)& \text{ with }\deg (\wt{\sigma}_0) \le n-1, \quad \text{ if } n \text{ odd} 
\end{cases} 
\end{align*}
and
\begin{align*}
(1+x)^n \left(\frac{1-x}{1+x}\right) \sigma_1\left(\frac{1-x}{1+x}\right) =\begin{cases}
(1-x^2) \,\wt{\sigma}_1(x) &\text{ with }\deg (\wt{\sigma}_1) \le n-2 , \quad \text{if } n \text{ even} \\
(1-x) \,\wt{\sigma}_1(x) &\text{ with }\deg (\wt{\sigma}_1) \le n-2,\quad \text{if } n \text{ odd}
\end{cases} \
\end{align*}

Note $\wt{\sigma}_0,\;\wt{\sigma}_1\in {\PO(1)}$. Then $f\in \PO(\mathcal{N}) $ and (2) is also satisfied. {Thus,} we are done in the case (\ref{1}).

\medskip
Using (\ref{thm.1}), we may assume that $\CS$ has atleast $2$ connected components.

\subsection{}\label{2.1}
\begin{enumerate}

\item[$(a_1)$]\label{5} If $\CS$ has a smallest element $a\in\partial \CS\subset F$ and no largest element, then $\CS\subset [a,\infty)=\CS(x-a)$. If $f\geq 0$ on $[a,\infty)$, then by (\ref{0 thm}), $f\in \PO(x-a)\subset \PO(\mathcal{N})$.

\item[$(a_2)$]\label{6} If $\CS$ has a largest element $b\in\partial \CS\subset  F$ and no smallest element, then $\CS\subset (-\infty,b]=\CS(b-x)$. If  $f\geq 0$ on $(-\infty,b]$, then by (\ref{0 thm}), $f\in \PO(b-x)\subset \PO(\mathcal{N})$.

\item[$(a_3)$]\label{7} If $\CS$ has a smallest element $a$ and a largest element $b$, then $a,b\in \partial \CS\subset F$ and $\CS\subset [a,b]=\CS(x-a,b-x)$. If $f\geq 0$ on $[a,b]$, then by (\ref{1}), $f\in \PO(x-a,b-x)\subset \PO(\mathcal{N})$.
\end{enumerate}

Therefore, {we may assume the following :} \\
\fbox{there exist $a,b\in \partial \CS\subset F$ with $(a,b)\cap \CS=\emptyset$ such that
$f(x)<0$ for some $x\in (a,b)$.} $(\clubsuit)$ 

\subsubsection{}
We make the following observation which will be used in the proof.\\
\fbox{If $h_1$, $h_2\in \Pos(\CS)$ {satisfy (\ref{thm.1}), then the product $h_1h_2$ also satisfies (\ref{thm.1}).}} $(\spadesuit)$ 

{ Assume}
 $$ h_1 = \sum_{e'\in \{0,1\}^s} \sigma_{e'} \, \ul g^{e'} ~~~\mbox{ and  }~~~ h_2 = \sum_{e''\in \{0,1\}^s} \tau_{e''} \, \ul g^{e''}$$ with $\sigma_{e'}, \tau_{e''} \in \PO(1)$, 
 $\deg (\sigma_{e'})\leq \deg (h_1)-\deg( \ul g^{e'})$ and $\deg (\tau_{e''})\leq \deg (h_2)-\deg( \ul g^{e''})$. Then we need to show that 
 $$ h_1h_2=\sum_{e\in \{0,1\}^s} \theta_e \,\ul g^e$$
 with $\theta_e \in \PO(1)$ and $\deg(\theta_e)\leq \deg(h_1h_2)-\deg(\ul g^e)$.
 Since 
 $$h_1h_2=\sum_{e',e''\in \{0,1\}^s} \sigma_{e'}\,\tau_{e''} \, \ul g^{e'}\, \ul g^{e''},$$ it is enough to show that $\ul g^{e'}\, \ul g^{e''}$ can be written as $\sigma \,\ul g^e$ for some $e=(e_1,\ldots,e_s)$ with $\deg(\sigma)\leq \deg(\ul g^{e'}\, \ul g^{e''})-\deg(\ul g^e)$. If $e'=(e_1',\dots, e_s')$ and $e''=(e_1'',\dots, e_s'')$, then take $e_i \equiv e_i'+e_i'' \pmod 2$ and $\sigma_e=\prod_{e_i'=e_i''=1}g_i^2$.
 This proves $(\spadesuit)$.

\subsection{}\label{thm.2} Assume the global minimum of $f$ over $\CS$ occurs at a point in $\CS^\circ$.

\subsubsection{}\label{2.2} First we assume that $f(c)=0$ for some  $c\in \CS^\circ$. 

Since $f\in \Pos(\CS)$, we get that $c$ is a global minimum of $f$ on $\CS$. Hence $f'(c)=0$ as $c\in \CS^\circ$. Note that there exist $\epsilon\in R_{>0}$ such that $f>0$ on $(c-\epsilon,c+\epsilon)\setminus \{c\}$. {Thus,} $c$ is a root of $f$ of even multiplicity $2r\geq 2$. Since $f\in F[x]$, we get that $c$ is algebraic  over $F$. If $q(x)\in F[x]$ is the minimal polynomial of $c$ over $F$, then $q^{2r}$ divides $f$ in $F[x]$. Let $f=q^{2r} h\text{ with } h\in F[x]$.

\smallskip
\noindent{\bf Claim 1 :} $h\in \Pos(\CS)$.
\smallskip

Assume claim 1 holds. Then by induction hypothesis on degree, $h\in \PO(\CN)$ and $h=\sum\sigma_e\ul g^e $ with $\sigma_e\in \PO(1) $ and $\deg (\sigma_e)\leq \deg(h)-\deg(\ul g^e)$. Hence $f\in \PO(\CN)$ and $f=\sum\wt{\sigma_e}\ul g^e $ with $\wt{\sigma_e}=\sigma_e (q^r)^2\in \PO(1)$ and $\deg (\wt{\sigma_e})\leq n-\deg(\ul g^e)$. This will prove the result in the case ~(\ref{2.2}). {Thus,} we need to prove claim 1.
\smallskip

\noindent{\bf Proof of Claim 1 :}
\begin{enumerate}
\item[$(b_1)$] Assume $\deg (q)=1$. Then $c\in F$ and $q=x-c$. For $a\in \CS\setminus \{c\}$, $f(a)\geq 0\implies h(a)\geq 0$. Since $(c-\epsilon,c+\epsilon)\subset \CS^\circ$, by continuity of $h$, we get $h(c)\geq 0$. Thus $h\in \Pos(\CS)$.

\item[$(b_2)$] Assume $\deg (q)>1$. Then $c\not\in F$. Let $c_1,\ldots,c_r $ be all the  roots of $q$ in $R$.

\begin{itemize}
\item[$(b_{2.1})$] If $a$ is an isolated point of $\CS$, then $a\in F$. Thus $a\neq c_j$ for all $j$. More generally, for any $b\in \CS\setminus \{c_1,\ldots,c_r\}$, we have $f(b)\geq 0$ and $q(b)\neq 0$. Hence $h(b)\geq 0$.

\item[$(b_{2.2})$] If {$c_j \in \CS$ for some $j$}, then $c_j$ is not an isolated point. Then $[c_j,c_j+\epsilon)\subset \CS$ or $(c_j-\epsilon,c_j]\subset \CS$ for some $\epsilon>0$. Thus $h\geq 0$ on $(c_j,c_j+\epsilon)$ or $(c_j-\epsilon,c_j)$. By continuity of $h$, we get $h(c_j)\geq 0$. Thus $h\in \Pos(\CS)$ and claim 1 is proved.
\end{itemize}
\end{enumerate}

\subsubsection{}\label{3.10} Assume $f> 0$ on $\CS$ and the global minimum of $f$ over $\CS$ is attained at some point of $\CS^\circ$. 

{\fbox{Assume $\deg(f)\leq 2$}. Then $f$ will not have a global minimum in $\CS^\circ$ (since $f$ satisfies $(\clubsuit)$).
Therefore,} this case will not arise. 

\fbox{Assume $\deg(f)=n\geq 3$.}  Note that $f$  has no zero in $\CS$ 
and by $(\clubsuit)$, $f$ takes a negative value at some point in $(\alpha,\beta)$, where $\alpha,\beta\in \partial \CS\cap F$ with $(\alpha,\beta)\cap \CS=\emptyset$. By change of variable, we may assume that $\alpha=-1$ and $\beta=1$. Thus, we have $-1,1\in \partial \CS$, $(-1,1)\cap \CS=\emptyset$ and  $f$ takes a negative value at some point in $(-1,1)$. Further, scaling $f$ by an element of $F_{>0}$,  we may assume that $f\geq 2$ on $\CS.$

\smallskip
 \noindent {\bf Idea of the proof :} 
We want to find $g\in \PO(\CN)$ which satisfies (2)  such that 
\begin{enumerate}
    \item[$(c_1)$] $\deg(g) \leq2< \deg(f)$,
     \item[$(c_2)$] $f-g\geq 0$ on $\CS$,
     \item[$(c_2)$] $f(r)=g(r)$ and $f(s)=g(s)$ for
    some $r,s\in F \cap (-1,1)$.
    \end{enumerate}
 Then $f-g=(x-r)(x-s)h$. Since $h\in \Pos (\CS)$, by induction hypothesis on degree, $h\in \PO(\CN)$ and $h=\sum\sigma_e\ul g^e $ with $\sigma_e\in \PO(1)$ and $\deg (\sigma_e)\leq \deg(h)-\deg(\ul g^e)$. By $(\spadesuit)$, we get the result for $f$. The proof for the existence of $g$ satisfying above is as follows.

Assume $f(c)=f(d)=0$, where $-1< c < d < 1$ and $f> 0$ on $[-1,c) \cup (d,1]$. 

$(d_1)$ If $c,d\in F$, then $f(x)=(x-c)(x-d)h$. As in the previous paragraph, with $g=0$, $r=c$ and $s=d$, we are done.

$(d_2)$ Assume at least one of $c,d \notin F$. Let $c<c'<c''<d$ be such that $f<0$ on $(c,c')\cup (c'',d)$. Choose an increasing sequence $(s_n)$ in $F\cap (c'',d)$ and a decreasing sequence $(r_n)$ in $F\cap (c, c')$ such that 
\begin{itemize}
\item[($d_{2.1})$] $s_n\to d$ and $r_n\to c$. 
\item[$(d_{2.2})$] Note that $f(r_n), f(s_n)<0$ and $f(r_n)\to f(c)=0$ and  $f(s_n) \to f(d)=0$.
\end{itemize}

If $c\in F$, then we choose $r_n=c$. Similarly, if $d\in F$, then we choose $s_n=d$.

\smallskip

\noindent{\bf{Claim 2 :}}
Given $\lambda\in R_{>0}$,  there exists a quadratic polynomial $G_\lambda \in \PO(\CN)$ such that
\begin{enumerate} 
    
    \item [$(d_3)$] $f-G_\lambda\geq 0$ on $[-\lambda,\lambda]\cap \CS$ and $G_\lambda\leq 1$  on $[-\lambda,\lambda]\cap \CS$.
    \item [($d_4)$] $f-G_\lambda=h' h$ for some $h'\in \PO(\CN)$ with $\deg(h')=2$ and $h\geq 0$ on $[-\lambda,\lambda]\cap \CS$.
\end{enumerate}

\noindent{\bf{Claim 3 :}} For {large $\lambda\in R_{>0}$}, both $f-G_\lambda$ and $h$ are $\geq 0$ on $\CS$.
\smallskip

Assume claims $2$ and $3$ hold. Then $f-G_\lambda= h'h$ with $h\in \Pos(\CS)$. By induction hypothesis on degree, $h\in \PO(\CN)$ and (2) is satisfied for $h$. Hence $f\in \PO(\CN)$ and (2) is satisfied for {$f$,} by $(\spadesuit)$. This proves the result in the case (\ref{3.10}). {Thus,} we need to prove claims $2$ and $3$.
\smallskip

\noindent{\bf Proof of Claim 2 :}
Fix $\lambda\in R_{>0}$. For $a_n,b_n\in F$ and $\delta_\lambda\in (0,1) \cap F$ (to be chosen later), let
    $$g_n(x)=\delta_\lambda [(x-r_n)(x-s_n)+ a_nx +b_n].$$
Solving for 
$  f(r_n)= g_n(r_n)$ and
  $  f(s_n)= g_n(s_n)$, we get 
$$    a_n= \frac{f(s_n)-f(r_n)}{\delta_\lambda(s_n -r_n)}  \:\: \mbox{and} \:\: b_n= \frac{-r_nf(s_n)+ s_nf(r_n)}{\delta_\lambda(s_n -r_n)}.$$
So, 
\[
g_n(x)=\delta_\lambda [(x-r_n)(x-s_n)]+ \frac{f(s_n)-f(r_n)}{(s_n-r_n)}x +\frac{-r_n f(s_n)+ s_n f(r_n)}{(s_n -r_n)}.\]

Since $s_n-r_n>c''-c'>0$, we get $\{\frac 1{s_n-r_n}\}$ is a bounded sequence. Further, $f(r_n)\to 0$ and  $f(s_n)\to 0$, so
$$\delta_\lambda\left[a_n(\pm \lambda)+b_n\right]\leq \frac 12$$
for large $n$. 
Let 
$$D_n= \max\{(\lambda+r_{n})(\lambda+ s_{n}),  (\lambda-r_{n})(\lambda-s_{n}) \}$$ We want to choose $\delta_\lambda \in (0,1) \cap F$ such that the following conditions hold:

\begin{enumerate}
    \item $\delta_\lambda D_n + \frac{1}{2} \leq 1$. \\
    Since $-1 < c \leq  r_n < s_n \leq d < 1 $, it is enough to choose $\delta_\lambda$ such that $\delta_\lambda \leq \frac{1}{2(\lambda+1)^2}$.
    \item $L(g_n)=\delta_\lambda<|L(f)|$, where $L(f)$ is the leading coefficient of $f$.
\end{enumerate}

As $-1 < c \leq  r_n < s_n \leq d < 1 $, there exists $\epsilon'>0$ such that for all $n\geq 1$, we get
$$( 1 +r_n)( 1 +s_n)> \epsilon' \mbox{~~ and ~~} ( 1 -r_n)( 1 -s_n) > \epsilon'.
$$ Further,
$a_n\to 0 $ and $ b_n \to 0$.
Therefore,  we get 
$g_n( \pm 1)\geq 0$ and $0 < g_n\leq 1$ on  $[-\lambda,\lambda]\cap \CS$ for  large $n$.

Let
$G_\lambda(x) =g_n(x)$ for large $n$ having above properties. Then $0<G_\lambda\leq 1$ on $[-\lambda,\lambda]\cap \CS$. Since $f \geq 2$ on  $\CS$, we get $f- G_\lambda \geq 0$ on $[-\lambda,\lambda]\cap \CS$.

We also have $f(r_n)=G_\lambda(r_n)$ and $f(s_n)=G_\lambda(s_n)$ and $\deg (G_\lambda)=2$. 
Therefore $$f(x)- G_\lambda(x)= h'h,$$
where $h'=(x-r_n)(x-s_n)\in \PO((x+1)(x-1))\subset \PO(\CN)$, by (\ref{CM}). Since $(f-G_\lambda)$ and $ h'$ are $> 0 $ on $[-\lambda,\lambda]\cap \CS$, we get $h\geq 0$ on $[-\lambda,\lambda]\cap \CS$. 
This completes the proof of claim 2.
\smallskip

\noindent{\bf Proof of Claim 3 :}
 Let $L(h)$ denote the leading coefficient of $h\in F[x]$.
\begin{enumerate}
\item[$(e_1)$] If $\CS$ is closed and bounded, then we can assume that $\CS\subset [-\lambda,\lambda]$. Thus $f-G_\lambda$ and $h$ are $\geq 0$ on $\CS$.

\item[$(e_2)$] If $\CS$ has a minimal element $a$ and contains a ray $[b,\infty)$, then for large $\lambda$ with $-\lambda< a$, we have $\CS\subseteq [-\lambda,\infty)$ and by claim 2, $f-G_\lambda\geq 0$ on $[-\lambda,\lambda]\cap \CS$. Since $L(G_\lambda)<L(f)$ and $\deg(G_\lambda)=2< \deg(f)$, therefore   $f-G_\lambda > 0$ on $ [\lambda, \infty)$. Thus $f-G_\lambda \geq 0$ on $\CS$.

\item[$(e_3)$] If $\CS$ has a maximal element $a$ and contains a ray $(-\infty,b]$. {Then by replacing $x$ by $-x$, we are in the case $(e_2).$}

\item[$(e_4)$] If $\CS$ has neither a minimal nor a maximal element and contains $(-\infty,b]\cup[a,\infty) $, then for large $\lambda$ with $L(G_\lambda)<L(f)$, we get $f-G_\lambda \geq 0$ on $\CS$.
\end{enumerate}

{Thus,} claim 3 is proved and so the result is proved in the case (\ref{3.10}).

Now, we will consider the last case when the global minimum of $f$ over $\CS$ is attained at some point of $\partial \CS$. This includes the case $\CS^\circ=\emptyset$ (i.e. $\CS$ is a finite set of points).
\medskip

By (\ref{2.2}), we may assume $f>0$ on $\CS^\circ$. 
We first consider the case when the global minimum of $f$ is attained at some $c \in \partial \CS$ with $f(c)>0$. 
Since $c,f(c)\in F$, we get $g=f-f(c)\in F[x]$ with  $g\geq 0$ on $\CS$ and $g(c)=0$. Since $f=g+f(c)$ and $f(c)\in F_{\geq0}$, it is enough to prove the theorem for $g$. Thus, it is enough to consider the case $f>0$ on $\CS^\circ$ and  $f(c)=0$ for some  $c\in \partial \CS.$

\subsection{} \label{thm.3} Assume $f>0$ on $\CS^\circ$ and the global minimum of $f$ is attained at some $c \in \partial \CS$ with $f(c)=0$.

\fbox{Assume $\deg(f)=1$.} Then $f$ does not satisfy $(\clubsuit)$, so this case does not occur.

\fbox{Assume $\deg(f)=2$.}
Since $f$ satisfies $(\clubsuit)$, $f'(c)\neq 0$ and $f$ is negative on one side of $c$.

\begin{itemize}
\item[$(f_1)$] Assume $f$ is negative on $(c, c+\epsilon)$ for some $\epsilon>0$. 
Then there exist $d\in \partial \CS$ with $c<d$ and $(c,d)\cap \CS =\emptyset$ such that $f(e)=0$ for some $e\in (c,d]$. 
Then $f=\beta (x-c)(x-e)$ for some $\beta\in F$ as $f\in F[x]$. Note $\beta>0$, since $f<0$ on $(c,e)$. Since $c\in F$ and $f\in F[x]$, we get $e\in F$.
By (\ref{CM}), $f\in \PO((x-c)(x-d)) \subset \PO(\CN)$ and satisfies (2). Thus, we are done.

\item[$(f_2)$] Assume $f$ is negative on $(c-\epsilon,c)$ for some $\epsilon>0$. The proof is similar to $(f_1)$.
\end{itemize}

\fbox{Assume $\deg(f)=n\geq 3$.}

Let us list the zeros of $f$ on $\CS$, i.e.
$$\{x\in \CS :f(x)=0\}=\{c_1<\ldots<c_r\}\subset \partial \CS.$$ We can find $\epsilon,\delta \in (0,1)$ such that $f>\delta$ on $\CS \setminus \cup_1^r (c_i-\epsilon,c_i+\epsilon)$, where {$(c_i-\epsilon,c_i+\epsilon)\cap \CS =\{c_i\},$} if $c_i$ is an isolated point of $\CS$. 

We will use the notations \fbox{$d_i',~d_i'',~c_i',~c_i''$} in the rest of the proof, which are explained below. These are elements of $\partial \CS$.
\begin{enumerate}
\item[$(g_1)$] If $c_i$ is  an  isolated point of $\CS$ and is not a maximal or a minimal element of $\CS$, then both $(g_{1.1})$  and $(g_{1.2})$  hold. When  $c_1$ is the minimal element of $\CS$, then only $(g_{1.1})$ {holds. When} $c_r$ is the maximal element of $\CS$, then only $(g_{1.2})$ holds.

\begin{itemize}
\item[$(g_{1.1})$] there exist $d_i''\in \partial \CS$ with $c_i<d_i''$ and $(c_i,d_i'')\cap \CS =\emptyset$. Note that if $f(d_i'')=0$, then $d_i''=c_{i+1}$ and $i<r$. So, $c_i<d_i''\leq c_{i+1}$.
\item[$(g_{1.2})$] there exist $d'_i\in \partial \CS$ with $d'_i<c_i$ and $(d'_i,c_i)\cap \CS =\emptyset$. Note that if $f(d'_i)=0$, then $d'_i=c_{i-1}$ and $i>1$. So, $c_{i-1}\leq d_i'< c_{i}$.
\end{itemize}
  
\item[$(g_2)$] If $c_i$ is not an isolated point of $\CS$, then exactly one of $(g_{2.1})$ and $(g_{2.2})$ {hold.}
\begin{itemize}
    \item[$(g_{2.1})$] Assume $c_i$ is a right boundary point of a connected component $[c'_i,c_i]$ of $\CS$ with $c'_i<c_i$ and $f>0$ on $(c'_i,c_i)\subset \CS^\circ$. If $c_i$ is not the maximal element of $\CS$, then there exist $d_i''\in \partial \CS$ with $c_i<d_i''$ and $(c_i,d_i'')\cap \CS=\emptyset$. Note that if $f(d_i'')=0$, then $d_i''=c_{i+1}$ and $i<r$. So, $c_{i}'<c_i< d''_i \leq c_{i+1}$.

    \item[$(g_{2.2})$] Assume $c_i$ is a left boundary point of a connected component $[c_i,c_i'']$ of $\CS$ with $c_i<c_i''$ and $f>0$ on $(c_i,c_i'')\subset \CS^\circ$.  If $c_i$ is not the minimal element of $\CS$, then there exist $d_i'\in \partial \CS$ with $d_i'<c_i$ and $(d_i',c_i)\cap \CS=\emptyset$. Note that if $f(d_i')=0$, then $d_i'=c_{i-1}$ and $i>1$. So, $c_{i-1}\leq d_i'<c_i< c_{i}''$.
\end{itemize}
\end{enumerate}

\smallskip
\noindent{\bf Claim ($\star$) :} There {exists} $g\in \PO(\CN)$ satisfying (2) and the following properties.
\begin{itemize}
\item[$(\star_1)$] $\deg(g)\leq  \deg(f)$. 
\item[$(\star_2)$] $f-g=h'h$ with $h'\in \PO(\CN)$ and $\deg(h')\geq 1$.
 \item[$(\star_3)$] $h\geq 0$ on $\CS$. 
\end{itemize}

Assume $(\star)$ holds. Then, by induction hypothesis, the theorem holds for $h$. Hence{ ,} it holds for $f=g+h'h${ ,} by $(\spadesuit)$.  Thus, we are reduced to proving $(\star)${ ,} which will be done in cases. Let $c_i\in\{c_1,c_2,\dots,c_r\}$. 

\subsubsection{}\label{t4.6} We will first consider the cases{ ,} when $(\star)$ can be proved with $g=0$.

\begin{enumerate}
\item[$(h_1)$] \label{positiveboth} There exists $i\in\{1,\ldots,r\}$ such that $f>0$ on both sides of $c_i$ (i.e. $f>0$ on $(c_i-\epsilon,c_i)\cup(c_i,c_i+\epsilon)$ for some $\epsilon>0$).

Since $f'(c_i)=0$, we can write $f(x)=(x-c_i)^2 h(x)$ with $h\in F[x]$. Clearly{ ,} $h\geq0$ on $\CS \setminus\{c_i\}$. Since $f$ and $(x-c_i)^2$ are $>0$ on $(c_i,c_i+\epsilon)$, we get  $h>0$ on $(c_i,c_i+\epsilon)$.  By continuity of $h$, we get $h(c_i)\geq0$. Hence $h\in \Pos(\CS)$. {Thus,} we get $(\star)$ with $g=0$ and $h'=(x-c_i)^2.$

\item[$(h_2)$] There exists $i\in\{1,\ldots,r\}$ such that $f<0$ on $(c_i,c_i+\epsilon)$  and $f<0$ on $(d_i''-\epsilon,d_i'')$ for some $\epsilon>0$. Note that $f(d_i'')=0$ (i.e. $d_i''=c_{i+1}$).

We can write $f(x)=(x-c_i)(x-d_i'')h(x)$ with $h\in F[x]$. Clearly{ ,} $h\geq0$ on $\CS \setminus\{c_i,d_i''\}$. Since $f$ and  $(x-c_i)(x-d_i'')$ are $<0$ on $(c_i,c_i+\epsilon)\cup(d_i''-\epsilon,d_i'')$, we get $h>0$ on $(c_i,c_i+\epsilon)\cup(d_i''-\epsilon,d_i'')$. By continuity of $h$, we get $h(c_i)\geq 0$ and  $h(d_i'')\geq 0$. Hence $h\in \Pos(\CS)$. {Thus,} we get $(\star)$ with $g=0$ and $h'=(x-c_i)(x-d_i'')\in \mathcal{N}$.

\item[$(h_3)$] \label{minimal}  $c_1$ is the minimal element of $\CS$ and $f>0$ on $(c_1,c_1+\epsilon)$ for some $\epsilon>0$.
    
    We can write $f(x)=(x-c_1)h(x)$. Clearly{ ,} $h\geq0$ on $\CS \setminus\{c_1\}$. Since $f$ and  $(x-c_1)$ are $>0$ on $(c_1,c_1+\epsilon)$, we get $h>0$ on $(c_1,c_1+\epsilon)$. By continuity of $h$, we get $h(c_1)\geq0$. Hence $h\in \Pos(\CS)$. {Thus,} we get $(\star)$ with $g=0$ and $h'=(x-c_1)\in \mathcal{N}$.
    
\item[$(h_4)$]  $c_r$ is the maximal element of $\CS$ and $f>0$ on $(c_r-\epsilon,c_r)$ for some $\epsilon>0$.
    
    We can write $f(x)=(c_r-x)h(x)$. Clearly{ ,} $h\geq0$ on $\CS \setminus\{c_r\}$. Since $f$ and $(c_r-x)$ are $>0$ on $(c_r-\epsilon,c_r)$, we get $h>0$ on $(c_r-\epsilon,c_r)$. By continuity of $h$, we get $h(c_r)\geq0$. Hence $h\in \Pos(\CS)$. {Thus,} we get $(\star)$ with $g=0$ and $h'=(c_r-x)\in \mathcal{N}$.
\end{enumerate}

\subsubsection{}\label{13.1} Now{ ,} we will prove $(\star)$ in the remaining cases. 

By (\ref{t4.6}),  we may assume that the {followings hold.}
\begin{enumerate}
\item[$(h_5)$]  $f\not > 0$ on $(c_i-\epsilon,c_i)\cup(c_i,c_i+\epsilon)$ for any $\epsilon>0$.

\item[$(h_6)$] If $f<0$ on $(c_i,c_i+\epsilon)$ for some $\epsilon>0$,
then $f\geq 0$ on $(d_i''-\epsilon',d_i'']$ for some $\epsilon'>0$.

\item[$(h_7)$] If $c_1$ is the minimal element of $\CS$, then $f<0$ on $(c_1,c_1+\epsilon)$ for some $\epsilon>0$. {In this case,} $c_1$ is an isolated point.

\item[$(h_8)$] If $c_r$ is the maximal element of $\CS$, then $f<0$ on $(c_r-\epsilon,c_r)$ for some $\epsilon>0$. {In this case,} $c_r$ is an isolated point.
    
\end{enumerate}

The main idea of the proof is as follows. For each $c_i$, we will find a 
$$g_i\in \{1,~~ g_i''= (x-c_i)(x-d_i''),~~ g_i'= (x-c_i)(x-d_i')\}\subset \PO(\CN)$$ such that $(\star)$ holds with $g=a g_1\dots g_r$ for some $a\in F_{>0}$. 
Further, we will ensure that $f-ag_1\ldots g_r$ vanishes at two points $r,s \in I\cap F$, where $I$ is  $[d'_1,c_1]$ or $[c_1,d''_1]$. {This will give 
$f-ag_1\ldots g_r=hh'$} with $h'=(x-r)(x-s)$  and $  h\geq 0$ on $\CS$.

Now{ ,} either $h'\in \PO((x-c_1)(x-d_1'))\subset \PO(\CN)$ or $h'\in \PO((x-c_1)(x-d_1''))\subset \PO(\CN).$ 
By induction hypothesis on degree, {the theorem holds for $h$. Hence,} it holds for $f${ ,} by $(\spadesuit)$. The choice of $g_i$ at $c_i$ will depend on  $f$ and also on the choice of $g_{i-1}$ at $c_{i-1}$, in the previous step. We will first choose $g_1$, then $g_2,\ldots$ and finally $g_r$. We will write $\epsilon,\; \epsilon'$ for small positive elements of $R$.

We will choose $g_i \in \{1,g_i'',g_i'\}$ at $c_i$ as follows:
\begin{enumerate} 
\item[$(i_1)$] If $c_i$ is an isolated point of $\CS$ and $f >0$ on $(c_i,c_i+\epsilon)\cup (d_i''-\epsilon,d_i'')$, then $f<0$ on $(c_i-\epsilon, c_i)$, by $(h_5)$. 
\begin{enumerate}
    \item If $d_i'\neq c_{i-1}$, then take $g_i=g_i'{ .}$
    \item If $d_i'= c_{i-1}$, then take $g_i=1$.
\end{enumerate}

\item[$(i_2)$]If $c_i$ is an isolated point of $\CS$, $f >0$ on $(c_i,c_i+\epsilon)$ and $f <0$ on $(d_i''-\epsilon,d_i'')$, then take $g_i=g_i''$.

\item[$(i_3)$] If $f <0$ on $(c_i,c_i+\epsilon)$ and $f >0$ on $(d_i''-\epsilon,d_i'')$, then take $g_i=g_i''$.

\item[$(i_4)$] If $f <0$ on $(c_i,c_i+\epsilon)\cup (d_i''-\epsilon,d_i'')$, then this case is $(h_2)$.

\item[$(i_5)$]  If $[c_i,c_i'']\subset \CS$, then $f<0$ on $(c_i-\epsilon, c_i)$, by $(h_5)$. Now
\begin{enumerate}
    \item if $d_i'\neq c_{i-1}$, then take $g_i=g_i'$.
    \item if $d_i'= c_{i-1}$, then take $g_i=1$.
\end{enumerate}

\end{enumerate}

Let us describe explicitly the choice of $g_i$'s. We will begin with $g_1$.

\begin{enumerate}

\item[$(A1)$] At $c_1$.
\begin{enumerate}

\item[$(j_1)$]Assume $c_1$ is the minimal element of $\CS$. 
Then $f<0$ on $(c_1,c_1+\epsilon)$, by $(h_7)$ (i.e.
$c_1$ is an isolated point). By $(h_6)$, $f\geq 0$ on $(d_1''-\epsilon, d_1'']$. Then take $g_1=g_1''$.

\item[$(j_2)$]Assume $c_1$ is not  the minimal element of $\CS$ and $f<0$ on $(c_1, c_1+\epsilon)$. Then $f\geq 0$ on $(d_1''-\epsilon, d_1'']$, by $(h_6)$. Then take  $g_1=g_1''$.

\item[$(j_3)$]Assume $c_1$ is not  the minimal element of $\CS$ and $f>0$ on $(c_1, c_1+\epsilon)$.
       \begin{itemize}
           \item[$(j_{3.1})$] \label{case1st} If $c_1$ is an isolated point of $\CS$, $f(d_1'')=0$ and $f<0$ on $(d_1''-\epsilon, d_1'')$, then take  $g_1=g_1''$.
          
           \item[$(j_{3.2})$] \label{case3.1} Take $g_1=g_1'$
           in the following cases. 
           \begin{enumerate}
               \item 
           $c_1$ is an isolated point of $\CS$ and $f\geq0$ on $(d_1''-\epsilon,d_1'']$ {(note $f(d_1')> 0).$}
           \item\label{case3.2} $[c_1, c''_1]\subseteq \CS$ {(note $f(d_1')> 0).$}
      \end{enumerate}
\end{itemize}
\end{enumerate}

\item[$(A2)$] At $c_i\notin\{c_1,c_r\}$.
\begin{enumerate}

    \item[$(j_4)$]Assume $c_i$ is an isolated point of $\CS$ and $f>0$ on $(c_i,c_i+\epsilon)$. 
    Then  $f<0$ on $(c_i-\epsilon,c_i)$, by $(h_5)$.
    
\begin{itemize}
    \item[$(j_{4.1})$] If $f\geq 0$ on $(d_i''-\epsilon,d_i'']$ and $f(d'_i)= 0 $, then take $g_i=1$.
    
    \item[$(j_{4.2})$] If $f\geq 0$ on $(d_i''-\epsilon,d_i'']$ and $f(d_i')> 0$, then take $g_i=g_i'$.

\item[$(j_{4.3})$]  If $f(d_i'')=0$ and $f<0$ on $(d_i''-\epsilon,d_i'')$, then take $g_i=g_i''$.
\end{itemize}

    \item[$(j_5)$]Assume $c_i$ is an isolated point of $\CS$ and $f<0$ on $(c_i,c_i+\epsilon)$. 
    Then $f\geq 0$ on $(d_i''-\epsilon, d_i'']$, by $(h_6)$. Take $g_i=g_i''$.

    \item[$(j_6)$]Assume $c_i$ is not an isolated point of $\CS$ and $[c_i,c_i'']\subseteq \CS$. 
    Then $f<0$ on $(c_i-\epsilon,c_i)$, by $(h_5)$.
\begin{itemize}
    \item[$(j_{6.1})$]  If $f(d_i')=0$, then $f>0$ on $(d_i',d_i'+\epsilon)$, by $(h_6)$. Take $g_i=1$.

    \item[$(j_{6.2})$]  If $f(d_i')>0$, then take $g_i=g_i'$.
    
\end{itemize}

\item[$(j_7)$]Assume $c_i$ is not an isolated point of $\CS$ and $[c_i',c_i]\subseteq \CS$. 
Then $f<0$  on $(c_i,c_i+\epsilon)$, by $(h_5)$.  By $(h_6)$, $f\geq 0$ on $(d_i''-\epsilon, d_i'']$. Take $g_i=g_i''$.

\end{enumerate}

\item[$(A3)$] At $c_r$.

\begin{enumerate}
    \item[$(j_8)$]Assume $c_r$ is the maximal element of $\CS$. Then $f< 0$ on $(c_r -\epsilon, c_r)$, by $(h_8)$.
       \begin{itemize}
         
           \item[$(j_{8.1})$] If $f(d_r')=0$, then $f>0$ on $(d_r', d_r'+\epsilon)$, by $(h_6)$. Take $g_r=1$.
           \item[$(j_{8.2})$] If $f(d_r')>0$, then take $g_r=g_r'$.
           
       \end{itemize}
    
    \item[$(j_9)$]Assume $c_r$ is not the maximal element of $\CS$ and $f> 0$ on $(c_r, c_r+ \epsilon)$ (note $f(d_r'')>0$). By $(h_5)$, $f< 0$ on $(c_r -\epsilon, c_r)$. 
\begin{itemize} 
\item[$(j_{9.1})$] Take $g_r=g_r'$ in the following cases.
\begin{enumerate} 
    \item $c_r$ is an  isolated point of $\CS$ and $f(d_r')>0$.
    
\item $[c_r, c''_r]\subseteq \CS$ and $f(d_r')> 0$.

\end{enumerate}
    \item[$(j_{9.2})$] Take $g_r=1$ in the following cases.
    \begin{enumerate} 
    \item $c_r$ is an  isolated point of $\CS$ and $f(d_r')= 0$.

   \item $[c_r, c''_r]\subseteq \CS$ and $f(d'_r)= 0$.
   
   \end{enumerate}
\end{itemize}
    
    \item[$(j_{10})$]Assume $c_r$ is not the maximal element of $\CS$ and $f< 0$ on $(c_r, c_r+\epsilon)$ 
(note $f(d_r'')> 0$). Take $g_r=g_r''$.

\end{enumerate}

\end{enumerate}
\medskip

For each $i=1,\ldots, r$, we have defined $g_i\in \PO(\CN)$. Let 
$$G=g_1\cdots g_r \in \PO(\CN)$$

\noindent{\bf Claim 4 : } We can choose
\begin{itemize}
\item[$(k_1)$] $b\in R$ from $(c_1,d_1'')$ or $(d_1',c_1)$ such that $f(b)=0$. 
\item[$(k_2)$] a sequence $(b_n)\in F$ converging to $b$ such that $f(b_n)$ and $G(b_n)$ have same sign and $|G(b_n)|> \epsilon$. 
 \item[$(k_3)$] Given $\lambda\in R_{>0}$, we can choose $n$ large such that 
$$G_\lambda(x)= \left(\frac {f(b_n)} {G(b_n)}\right) G(x)< \delta\hspace{.1in} \text{on } [-\lambda,\lambda]\cap \CS
.$$

\end{itemize}

\noindent{\bf{Claim 5 :}} 
\begin{enumerate}
    \item [$(k_4)$] $\deg(G_\lambda)\leq \deg(f){ .}$ 
    \item [$(k_5)$] $f-G_\lambda\geq 0$ on $[-\lambda, \lambda]\cap \CS$. 
     \item [($k_6)$] $f-G_\lambda=h' h$ for some $h'\in \PO(\CN)$ with $\deg(h')\geq 1$ and $h\geq 0$ on $[-\lambda,\lambda]\cap \CS$.
\end{enumerate}
\smallskip 

\noindent{\bf{Claim 6 :}} For appropriate choice of $\lambda$, both $f-G_\lambda$ and $h$ are $\geq 0$ on $\CS$.
\smallskip

Assuming claims $4, 5$ and $6$, we will complete the proof of (\ref{thm.3}). For appropriate choice of $\lambda$, $f-G_\lambda=h'h$ for some $h'\in \PO(\CN)$ with $\deg(h')\geq 1$ and $h\geq 0$ on $\CS$. With $g=G_\lambda$, we get $(\star)$. Thus, it remains to prove claims $4, 5$ and $6$.

\smallskip
\noindent{\bf Proof of Claim 4 :} 
\begin{enumerate}

\item[$(l_1)$]\label{c1} 
Assume $c_1$ is not  the minimal element of $\CS$ and $f>0$ on $(c_1, c_1+\epsilon)$.
\begin{enumerate}
    \item[$(l_{1.1})$] If $c_1$ is an isolated point, $f(d_1'')=0$ and $f<0$ on $(d_1''-\epsilon, d_1'')$, then there exists $b\in (c_1,d_1'')$ such that $f(b)=0$ and $f<0$ on $(b,b+\epsilon')$. 
Choose a decreasing sequence $b_n\in F$ converging to $b$ such that $f(b_n)<0$. Note that $G(b)\neq 0$, so we may assume $G(b_n)\neq 0$, hence $|G(b_n)|> \epsilon$. Note that $f(b_n)$ and $G(b_n)$ both are $<0$. Since $\frac {f(b_n)}{G(b_n)}\to 0$, we get 
\[G_\lambda(x)= \frac {f(b_n)} {G(b_n)} G(x)< \delta \hspace{.1in}\text{on}\hspace{.1in} [-\lambda,\lambda]\] for large $n$.

\item[$(l_{1.2})$] {Assume  either} $c_1$ is an isolated point and $f\geq0$ on $(d''_1-\epsilon,d''_1]$ or $[c_1,c_1'']\subseteq \CS$. Then{ ,} as we explained in $(j_{3.2})$, we have $f(d'_1)>0$, i.e. $f>0$ on $[d'_1,d'_1+\epsilon)$ and $f<0$ on $(c_1-\epsilon,c_1)$. Then there {exists} $b\in (d'_1,c_1)$ such that $f(b)=0$ and $f<0$ on $(b,b+\epsilon')$. Choose a decreasing sequence $b_n\in F$ converging to $b$ such that $f(b_n)<0$. Rest of the proof is as in $(l_{1.1})$.

\end{enumerate}

\item[$(l_2)$] 
If $f<0$ on $(c_1, c_1+\epsilon)$, then  by $(h_6)$, $f\geq0$ on $(d''_1-\epsilon,d''_1]$. Then there exists $b\in (c_1,d_1'')$ such that $f(b)=0$ and $f<0$ on $(b-\epsilon',b)$. Choose an increasing sequence $b_n\in F$ converging to $b$ such that $f(b_n)<0$. 
Rest of the proof is as in $(l_{1.1})$. 

\item[$(l_3)$]
 Assume $c_1$ is  the minimal element of $\CS$ and $f>0$ on $(c_1, c_1+\epsilon)$. This case is covered by $(h_3)$.
\end{enumerate}

Thus, claim 4 is proved. 
\smallskip

\noindent{\bf Proof of Claim 5 :} Note that $f>0$ on $\CS^\circ$.

 \begin{enumerate}
\item [($l_4$)]
For every $c_i$, we have chosen a $g_i\in \{1, g'_i,g''_i\}$. 
\begin{enumerate}
    \item[$(l_{4.1})$] $g_i=1$ does not contribute to the degree of $G_N$.
    
    \item[$(l_{4.2})$] If $g_i=g'_i=(x-c_i)(x-d'_i)$, then $f$ has at least one root $b'\in(d'_i,c_i)$.
    
    \item[$(l_{4.3})$] If $g_i=g''_i=(x-c_i)(x-d''_i)$, then $f$ has at least one root $b''\in(c_i,d''_i)$.
\end{enumerate} 

If $f$ has atmost $2$ consecutive roots in the set $\partial \CS$, then by construction of $g_i$, $(k_4)$ follows. 

If $f$ has $\geq 3$ consecutive roots in the set $\partial \CS$, say $c_{i-1},c_{i},c_{i+1}$, then if  $g_i=1$, then we are done. For example, this will be the case when $f>0$ on $(c_i,c_i+\epsilon)\cup (d_i''-\epsilon,d_i'')\cup (d_i',d_i'+\epsilon)$ and  $f<0$ on $(c_i-\epsilon,c_i)$ with $d_i'=c_{i-1}$ and $d_i''=c_{i+1}$. If $g_{i}=(x-c_i)(x-c_{i+1})$ and $g_{i-1}=(x-c_{i-1})(x-c_i)$, then $f$ has  at least two roots $b,\,b'$ such that $b \in (c_{i-1}, c_{i})$ and $b' \in (c_{i}, c_{i+1})$.  So, $(k_4)$ follows.

\item [($l_5$)] 
Since $G(b_n)\to G(b)\neq0$ and $\frac {f(b_n)}{G(b_n)}\to 0$, choosing $n$ large, $G_\lambda$ will satisfy $(k_5)$.

\item [($l_6$)]
Since $f(b_n)=G_\lambda(b_n)$, we get $f(x)-G_\lambda(x)=h'h$ with $h'=(x-b_n)(x-t)$ for $t\in\{c_1,d_1''\}$,  and $h\geq0 $ on $[-\lambda,\lambda]\cap \CS$. Since $h'\in \PO(\CN)$, by (\ref{CM}), $(k_6)$ follows.
\end{enumerate}
   Thus, claim 5 is proved.
   
\smallskip
\noindent{\bf Proof of Claim 6 :}
     This is similar to the proof of claim 3, hence we omit it.
    
     Thus, the proof of (\ref{1 thm}) is complete. $\hfill\square$
\end{proof}

\begin{figure}[htbp]
    \centering
     \includegraphics[width=\linewidth]{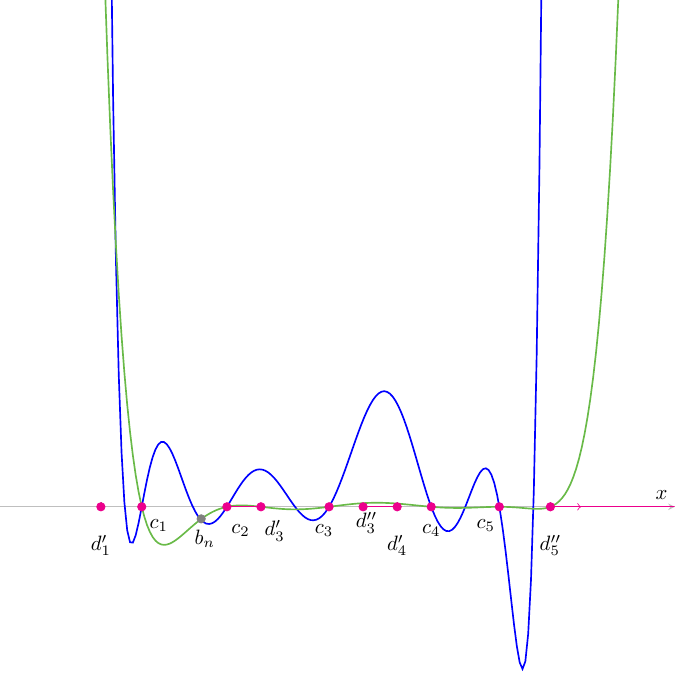} 
  Here \raisebox{0.5ex}{\textcolor{blue}{\rule{0.7cm}{1pt}}} denotes $f(x)$, \raisebox{0.5ex}{\textcolor{magenta}{\rule{0.7cm}{1pt}}} denotes $\CS$ and \raisebox{0.5ex}{\textcolor{dpg}{\rule{0.7cm}{1pt}}} denotes $G_N(x)$.\\[1pt] 
  
$g_1=(x-c_1)(x-c_2)$,~ $g_2=1$, ~$g_3=(x-c_3)(x-d'_3)$, ~$g_4=(x-c_4)(x-c_5)$, ~$g_5=(x-c_5)(x-d_5'')$.
 \caption{An example of 3.4.2}
  \label{fig}
\end{figure}

\smallskip

As a consequence of (\ref{1 thm}), we extend some results from Kuhlmann and Marshall  \cite{Sal1}, where {they are} proved for $F=\BR$.

 \begin{lemma}
 Let $\ul g\subset F[x]$ such that  $\Pos(\CS(\ul g))=\PO(\ul g)$ (i.e. the preordering $\PO(\ul g)$ is saturated). Then each $f\in \Pos(\CS(\ul g))$ 
can be written as
$f = \sum_{e \in \{0,1\}^s} \tau_e \,\ul g^e$,
where 
${\tau_e \in \PO(1)}$ with  $\deg(\tau_e \,\ul g^e)\leq \deg(f) + N$ for some $N$ depending on $g_1, \ldots, g_s$ only.
\end{lemma}
  
\begin{proof}
Follow the proof of  \cite[Cor 4.2]{Sal1} and use (\ref{1 thm}).
$\hfill \square$
\end{proof}

 \begin{theorem}\label{Non cpt sat}
 Let $\ul g\subset F[x]$ such that {$\CS=\CS(\ul g)$} is an unbounded closed semi-algebraic set in $R$ with $\partial \CS\subset F$. Let $\mathcal{N}$ be natural choice of generators of $\CS$. Then
 \begin{enumerate}
     \item  $\Pos(\CS)=\PO(\ul g)$ (i.e. the preordering $\PO(\ul g)$ is {saturated}) if and only if $\ul g$ contains all the natural choice of generators of $\CS$, upto scaling by $F_{> 0}$.
     
     \item  $\Pos(\CS)=\QM(\CN)$ (i.e.  
     the quadratic module $\QM(\CN)$ is {saturated}) if and only if either $|\CN|\leq 1$ or $|\CN|=2$ and $\CS$ has an isolated point. 
 \end{enumerate}
 \end{theorem}

\begin{proof}
(1) If $\ul g$ contains the natural choice of generators of $\CS$, {then by (\ref{1 thm}), $\Pos(\CS)=\PO(\ul g)$}.
Assume $\Pos(\CS)=\PO(\ul g)$. Since $\CS(\ul g)$ is unbounded, replacing $x$ by $-x$, if necessary, we can assume that $[a,\infty) \subset \CS$ for some $a\in F$. 
Thus, all elements of $\ul g$ have positive leading coefficients. The proof given by Kuhlmann and Marshall \cite[Thm 2.2]{Sal1} works in our case, so we will not repeat it. 
But we will give a simpler argument to prove the case that  $p(x)=(x-a)(x-b)\in \ul g$, upto scaling by $ F_{>0}$, where $a< b$, $a,b \in \partial \CS$ and $(a,b) \cap \CS = \emptyset$.  

Since $p \geq 0 $ on $\CS$, we get {$p\in \PO(\ul g)$, by assumption $\Pos(\CS)=\PO(\ul g)$}. Since  $\deg(p)=2$, we can write 
{$$p=  \sigma  + \sum \sigma_jg_j + \sum \sigma_{kl}\;g_k g_l $$} where 
\begin{itemize}
    \item[$(m_1)$] {$\sigma \in \PO(1)$} has degree $0$ or $2$.
    \item[$(m_2)$] $\sigma_j \in F_{\geq 0}$ and
$g_j\in \ul g$ is linear or quadratic.
    \item[$(m_3)$]  $\sigma_{kl} \in  F_{\geq 0}$  and $g_k, g_l \in \ul g$ are linear.
\end{itemize}
   If $g_i \in \ul g$ is linear, then it is increasing as $[a,\infty) \subset \CS$ and $g_i(a) \geq 0$ gives $g_i>0$ on $(a,b)$. Thus
$$p\geq \sigma_{i_1} g_{i_1} + \cdots + \sigma_{i_r}g_{i_r}~~  \text{on}~~ (a, b)$$ 
where $g_{i_j}$  are {those quadratics} in $\ul g$ which assume
some negative value on $(a , b)$ and $\sigma_{i_j}\in F_{>0}$. Note that $r\geq 1$, since $p<0$ on $(a,b)$.  Since $g_{i_j}(a)\geq 0$ and $\sum \sigma_{i_j}(a) g_{i_j}(a)  \leq p(a)=0$, we get $g_{i_j}(a)=0$ for all  $j$. {Similarly,} $p(b)=0$ gives  $g_{i_j}(b)=0$ for all $j$. Therefore, {$g_{i_j}(x)=\alpha(x-a)(x-b)$,
where $\alpha\in F_{>0}$.} So we are done.

(2) Follow the proof of \cite[Thm 2.5]{Sal1} and use (\ref{1 thm}).
$\hfill \square$
\end{proof}

\section{Applications of  {(\ref{1 thm})} : Saturation of quadratic module $\QM(\CN) $ in $\BF[x]$}

In this section, all the results are for archimedean field $\BF$, since the proofs {use} results of Schm\"{u}dgen and Putinar, which are not valid for non-archimedean $R$.

\begin{theorem}\label{cpt sat}

Let $\ul g\subset \BF[x]$ such that $\CS =  \CS(\ul g) = \bigcup_{j=0}^{k} [a_j, b_j]\subseteq \BR$ is compact and $\partial \CS\subset \BF$ with $ b_{j-1} < a_j$. Then $\Pos(\CS)=\PO(\ul g)$ if and only if the following two conditions hold:
\begin{enumerate}
    \item For each endpoint $a_j$, there exists $i_j \in \{1, \dots, s\}$ such that $g_{i_j}(a_j) = 0$ and $g_{i_j}'(a_j) > 0$,
    \item For each endpoint $b_j$, there exists $i_j \in \{1, \dots, s\}$ such that $g_{i_j}(b_j) = 0$ and $g_{i_j}'(b_j) < 0$.
\end{enumerate}
\end{theorem}

\begin{proof}
Follow the proof of \cite[Thm 3.2]{Sal2} and use {(\ref{1 thm},  \ref{1st cor})}. $\hfill\square$
\end{proof}

\begin{corollary}\label{T.4.4}
Let {$\ul g \subset \BF[x] $ such that $\CS =  \CS(\ul g)\subset \BR$ is compact, $\partial\CS\subset \mathbb{F}$ and $\CS$ has no isolated point.} Then $\Pos(\CS)= \PO(\ul g)$  if and only if for each $a \in \partial\CS$, there exists $i$ such that $g_i(a)=0$ and $g_i'(a)\neq 0$ (i.e. $(x-a)|g_i$ and $(x-a)^2\not| g_i$).
\end{corollary}

\begin{proof}
$(\Leftarrow)$ follows from (\ref{cpt sat}). 
$(\Rightarrow)$ If $[a,b]\subset \CS$ with $a<b$ and $a,b\in \partial \CS$, then there {exists} $i$ such that $g_i(a)=0$ and $g_i'(a)\neq 0$. Since $g_i>0$ on $(a,a+\epsilon)$ for some $\epsilon>0$, we get $g_i'(a)>0$. Similarly,   there {exists} $j$ such that $g_j(b)=0$ and $g_j'(b)\neq 0$. Since $g_j>0$ on $(b-\epsilon,b)$ for some $\epsilon>0$, we get $g_j'(b)<0$. Now{ ,} use (\ref{cpt sat}). $\hfill \square$ 
\end{proof}

\begin{remark}
    {In (\ref{T.4.4}), we need the condition that $\CS(\ul g)$ has no isolated point.} To see this, take $\ul g=\{x,-x^2\}$, so $\CS(\ul g)=\{0\}$. Then $-x\in \Pos(\CS(\ul g))$.
Assume $-x\in \PO(\ul g)$. Then
    $$-x= \sigma_0+ \sigma_1x+\sigma_2(-x^2)+\sigma_3(-x^3)$$ with  $ \sigma_i \in \PO(1)$. It gives $\sigma_0(0)=0$.  So $\sigma_0= x^2 \sigma_0 '$ with $ \sigma_0' \in \PO(1)$. Now
   $-1= x\sigma_0 ' + \sigma_1 + \sigma_2(-x)+ \sigma_3(-x^2)$
gives $-1=\sigma_1(0)$, a contradiction. Hence 
   $-x\notin \PO(\ul g)$. Here $g=x\in \ul g$ such that $g'(0)\neq 0$, so condition of (\ref{T.4.4}) holds. But there is no $g\in \ul g$ such that $g(0)=0$ and $g'(0)<0$. So (\ref{cpt sat}) does not hold.
$\hfill \square$
\end{remark}

\begin{corollary}\label{cor.2}
Let $\ul g \subset \BF[x] $ such that $\CS= \CS(\ul g)\subset \BR$ is {compact, $\partial\CS\subset\BF$ and $\CS$ has no isolated point.} Let $\mathcal{N}$ be the natural choice of generators of $\CS$ and {$\pi= \prod_{g\in \mathcal N} g$.} Then
$\QM(\pi) = \QM({\mathcal{N}}) = \PO(\CN) = \PO(\pi).$
\end{corollary}

\begin{proof}
Follow the proof of  \cite[Cor 3.4]{Sal2} and use (\ref{T.4.4}, \ref{prime}).
$\hfill \square$
\end{proof}


\begin{theorem}\label{Archi1}
    Let $\ul g\subseteq \BF[x]$ such that $\CS(\ul g)\subset\BR$ is compact. Then $\QM(\ul g)$ is an archimedean quadratic module of $\BF[x]$.
\end{theorem}

\begin{proof}
Follow the proof of \cite[Thm 7.1.2]{Mur1}, where it is proved in the case $\BF=\mathbb{R}$. 
    First one proves that there exists $f\in \QM(\ul g)$ such that $\deg(f)$ is even and {the leading coefficient} of $f$ is negative. Since $\CS(f)$ is compact, by (\ref{Pow thm1}), we have  $\PO(f)$ is an archimedean preordering of $\BF[x]$. Since $\PO(f)\subset \QM(\ul g)$, so $\QM(\ul g)$ is an archimedean quadratic module of $\BF[x]$.
    $\hfill \square$
\end{proof}

\begin{theorem}\label{M=T}
    Let $\ul g\subseteq \BF[x]$ such that $\CS=\CS(\ul g)\subset\BR$ is compact with $\partial \CS \subset \BF$. Let $\mathcal{N} \subset \BF[x]$
   be  the  natural choice of generators of $\CS$. Then $\QM({\mathcal{N}}) = \PO(\CN)=\Pos(\CS)$.
\end{theorem}

\begin{proof} 
Follow the proof of \cite[Thm 3.5]{Sal2}.
By (\ref{1 thm}), $\Pos(\CS)=\PO(\CN)$. To show $\PO(\CN)= \QM(\CN)$, we need to show that if $g_i,g_j\in \mathcal N$, then $g_ig_j\in \QM(\CN)$.

Case 1: Use (\ref{cor.2}).

Case 2: When $a=c$, use (\ref{Non cpt sat}(2)).  For $a<c$, choose $\beta \in \BF$  such that $\beta - x > 0$ on $\CS$. By (\ref{Archi1}), $\QM(\CN)$ is archimedean. Hence, by (\ref{Pow thm2}), we have $\beta-x\in \QM(\CN)$. Apply (\ref{cor.2}) to $[a,c]\cup [d,\beta]${, to get} $g_ig_j\in \QM(\CN)$.

Case 3: When $c'=d$, use (\ref{Non cpt sat}(2)).  For $c'>d$, choose $\alpha,\beta \in \BF$  such that $x-\alpha,\beta - x > 0$ on $\CS$. By  (\ref{Archi1}), $\QM(\CN)$ is archimedean. Hence, by (\ref{Pow thm2}), we have $x-\alpha,\beta-x\in \QM(\CN)$. Apply (\ref{cor.2}) to $[\alpha,c]\cup [d,c']\cup[d',\beta]${, to get} $g_ig_j\in \QM(\CN)$.
    $\hfill\square$
\end{proof}

\begin{question}\label{q.7.3}
Let $\CS \subset \BR$ be a closed semi-algebraic set with $\partial \CS\subset \BF$. Let  $\mathcal{N}=\{g_1, \ldots, g_s\}\subset \BF[x]$ be the natural choice  of generators of $\CS$. By (\ref{1 thm}), $\Pos(\CS)=\PO(\CN)$. We can ask the following
\begin{enumerate}
    \item Is $\Pos(\CS)=\QM({\mathcal{N}})$ (i.e. whether $\QM(\CN)$ is saturated)?
    \item Let $f\in \QM(\CN) $. Does there exists  a representation $f = \sigma_0 + \sigma_1 g_1 + \cdots + \sigma_s g_s$ with $\sigma_i \in \PO(1)$ such that $\deg(\sigma_i g_i)$ is bounded as a function of $\deg(f)$ (i.e. whether $\QM(\CN)$ is stable with respect to the system of generators $\CN$)? 
\end{enumerate}
\end{question}

{\bf Answer:} 
\begin{enumerate}
    \item Assume $s \leq 1$. Then $\Pos(\CS)=\PO(\CN)=\QM(\CN)$, by (\ref{1 thm}(1)) and $\deg(\sigma_ig_i)\leq \deg(f)$, by (\ref{1 thm}(2)). 
    
    \item Assume $s=2$, $\CS$ is non-compact and has an isolated point. Then $\Pos(\CS)=\PO(\CN)=\QM(\CN)$, by (\ref{Non cpt sat}(2)). Further, following the proof of \cite[Thm 2.5]{Sal1}, we get  $\deg(\sigma_ig_i)\leq \deg(f)$.
    
   \item Assume $\CS$ is non-compact and either (i) $s=2$ and $\CS$ has no isolated point or (ii) $s\geq 3$. Then $\QM(\CN)\neq \PO(\CN)$, by (\ref{Non cpt sat}(2)), thus answering Qn \ref{q.7.3}(1) in negative. We do not know Qn \ref{q.7.3}(2) in this case.

    \item Assume $s=2$ and $\CS$ is compact. Then following the proof of \cite[Thm 3.5]{Sal2}, we get $\deg(\sigma_ig_i)\leq \deg(f)+1$.

\item Assume $s \geq 3$ and $\CS$ is compact. Then there is no such degree bound. 
For an example, we may assume that $a$ is the smallest element of $\CS$ and $d$ is the largest element of $\CS$. Let $g_1 = (x - a)$, $g_2 = (x - b)(x - c)$, where $a < b < c<d$ be elements of $F$. It is shown in \cite[Example 4.4]{Sal2} that for $\beta > d$, if we consider
the representation
$$g_1 g_2 = \sigma_0 + \sigma_1 g_1 + \sigma_2 g_2 + \sigma_3 (\beta - x)$$
where $\sigma_i \in \sum \mathbb{R}[x]^2$, then as $\beta\to \infty$, $\max\{\deg(\sigma_i):i=0,\ldots,3\}\to \infty$. Thus{ , there is no bound} on $\deg(\sigma_i g_i)$ as a function of $\deg(f)$. Note that  $\Pos(\CS)=\PO(\CN)=\QM(\CN)$, by (\ref{1 thm}(1), \ref{M=T}).
$\hfill \square$
\end{enumerate}


\section{Analogues of { (\ref{1 thm})} when $\partial\CS\not\subset F$}

Let $\ul g \subset F[x]$ and $\CS = \CS(\ul g)\subset R$ be a closed semi-algebraic set. {In (\ref{1 thm}), under the assumption $\partial \CS\subset F$, we proved that $\Pos(\CS)=\PO(\CN)$,} where $\CN$ is the natural choice of generators of $\CS$. In this section, we will discuss some cases when $\partial \CS\not\subset F$. Note that if $c\in \CS$ and $g_i(c)>0$ for all $i$, then $(c-\epsilon,c+\epsilon)\subset \CS$ { for some $\epsilon>0$. Hence} $c\not\in \partial \CS$. Thus,
if $c\in \partial \CS$, { then $g_i(c)=0$ for some $i$.} Therefore, all elements $c\in \partial \CS$ are algebraic over $F$. 

The following result generalizes (\ref{CM}).

\begin{lemma}\label{CM2}
 Let $a,b\in R $ with $g= (x-a)(x-b)\in F[x]$ (i.e. $a+b, ab\in F$). Let $c,d\in F$ with $a\leq c<d\leq b$.  Then $(x-c)(x-d)\in \PO(g)$.
\end{lemma}

\begin{proof}
In case $a,b\in F$, we are done by (\ref{CM}). So we { assume  $a,b\not\in F$.}
Further, $$(x-a)(x-b)=\left(x-\frac{a+b}2\right)^2-\alpha \hspace{.2in}\text{with}\hspace{.2in} \alpha=\left(\frac{a+b}2\right)^2-ab\in F_{> 0}.$$
Changing the variable $x$ to $x+\frac {a+b}2$, we are reduced to the case that $a=-\sqrt \alpha$ and $b=\sqrt \alpha$ with $\alpha\in F_{>0}$.
 { Thus,} we have $-\sqrt \alpha\leq c<d\leq \sqrt \alpha$ and we need to { show  $(x-c)(x-d)\in \PO(x^2-\alpha)$.}

Note that if $c=-\sqrt \alpha$, then $d=\sqrt \alpha$, otherwise $(x-c)(x-d)\not\in F[x]$. { So,} we assume $c\neq -\sqrt \alpha$ and $d\neq \sqrt \alpha$.
Choose $r\in F_{>0}$ such that $-\sqrt \alpha<-r<c<d<r<\sqrt \alpha$ and  $\frac {r^2}{\alpha}\in(0,1)\cap F_{>0}$.
Since $(x-c)(x-d)\in \PO(x^2-r^2)$, by (\ref{CM}), it is enough to show $x^2-r^2\in \PO(x^2-\alpha)$.
Note that  
$$(x^2-r^2)-\frac {r^2}{\alpha}(x^2-\alpha)=\left(1-\frac {r^2}\alpha\right)x^2\in \PO(1){,}$$ hence $x^2-r^2\in \PO(x^2-\alpha)$.
$\hfill \square$
\end{proof}

\begin{theorem}\label{t4.1}
Let $\ul g \subset F[x]$ such that $\CS = \CS(\ul g)\subset R$ is unbounded. Assume there exists $c\in \partial \CS$ with $[F(c):F]\geq 3$. Then 
$\Pos(\CS)$ is not a finitely generated preordering of $F[x]$.
\end{theorem}

\begin{proof}
   Since $\CS$ is not bounded, wlog, we may assume $[a,\infty) \subset \CS$ for some $a\in F$.  Note that any $f\in \Pos(\CS)$ has a positive leading coefficient. Assume $\Pos(\CS)= \PO(f_1,\ldots, f_t,f_1',\ldots,f_{t'}')$ is finitely generated with $f_1,\ldots,f_t$ of degree $\leq 2$ {and $\deg f_i'\geq 3$}. Then $f_i(c)>0$ for $i=1,\ldots,t$, since $[F(c):F]\geq 3$. 
   There exists $\epsilon>0$ such that $\CS \cap(c-\epsilon,c)=\phi$ or $\CS \cap(c,c+\epsilon)=\phi$ and $f_i>0$ on  $(c-\epsilon,c+\epsilon)$ for $i=1,\ldots,t$. 
    Choose  $r<s$ in $F\cap (c-\epsilon,c)$ or $F\cap (c,c+\epsilon)$.
  Let $p=(x-r)(x-s)\in \Pos(\CS)$. 
  By assumption, {$\Pos(\CS)= \PO(f_1,\ldots, f_t, f_1',\ldots, f_{t'}')$,} hence
  $$p= \sigma_0 + {\sum_{i=1}^t} \sigma_i f_i + {\sum_{1\leq j<k\leq t}} \sigma_{jk}f_j f_k $$
  where $f_i$'s are linear or quadratic and $f_j,f_k$'s are linear. Note that $f_i'$ will not appear in the representation of $p$, as $\deg(f_i')\geq 3$ and all polynomials in $\Pos(\CS)$ have a positive leading coefficient. Now $p(r)=0$ and $f_i(r)>0$ for $i=1,\ldots,t$. Thus $\sigma_0(r)=\sigma_i(r)=\sigma_{jk}(r)=0$ for all $i,j,k$. Thus $(x-r)^2$ divides $\sigma_0, \sigma_i, \sigma_{jk}$ in $F[x]$. Hence $(x-r)^2$ divides $p$, a contradiction.
Thus $\Pos(\CS)$ is not a finitely generated preordering of $F[x]$.
$\hfill\square$
\end{proof}

\begin{proposition}\label{t4.4}
  Let $\CS=[\sqrt[3]2,\infty)=\CS(x^3-2)\subset \BR$. Then
    $\Pos(\CS)= \CP =  \cup_{n\geq 0} \PO(x-r_n,x^3-2)$ in $\BQ[x]$, where $(r_n)$ is any increasing sequence in $\BQ$ converging to $\sqrt[3]2$. 
  \end{proposition}
  
\begin{proof}
By (\ref{t4.1}), $\Pos(\CS)$ is not a finitely generated preordering {of $\BQ[x]$}.
Let $f\in \BQ[x] \cap \Pos(\CS)$. If $f\geq 0$ on $\BR$, then $f\in \sum \BQ[x]^2$. Assume $f$ takes a negative value at some point in $(-\infty,\sqrt[3]2)$.
If $f(\sqrt[3]2)=0$, then $f=(x^3-2)h$ with $h\geq 0$ on $\CS$. By induction on degree, $h\in \CP$, hence $f\in \CP$.
Assume $f(\sqrt[3]2)>0$. For large $n$, $f\geq 0$ on $[r_n,\infty)$. Thus $f\in \PO(x-r_n)\subset \CP$ {by (3.1),} hence $\Pos(\CS)=\CP$.
$\hfill \square$
\end{proof}

\begin{theorem}\label{t4.5}
Let  $\ul g \subseteq F[x]$ and  $\CS = \CS(\ul g)\subset R$ with $[F(c):F]\leq 2$ for all $c\in \partial \CS$. Assume that if $[F(c):F]=2$ for some $c\in \partial \CS$ and $c<c'$ are conjugates over $F$, then $c'\in \partial \CS$ and 
\begin{enumerate} 
\item either $\CS=[c,c']$ and $-(x-c)(x-c')\in \ul g$ 
\item or $(c,c')\cap \CS=\emptyset$ and $(x-c)(x-c')\in \ul g$. 
\end{enumerate}
Then $\Pos(\CS)=\PO(\ul g)$.
\end{theorem}

\begin{proof}
\fbox{{\it Case 1.}} Let $\CS=(-\infty,-\sqrt d] \cup [\sqrt d,\infty)=\CS(x^2-d)$, where $d \in F_{>0}$ and $\sqrt{d}\notin F$. Note that every $f\in F[x] \cap \Pos(\CS)$ has even degree with positive leading coefficient.
 If $f\geq 0$ on $R$, then $f\in \PO(1)$, by (\ref{sos}). Assume $f$ takes a negative value at some point in $(-\sqrt{d}, \sqrt{d})$. Let $c_1< c_2\in R$ be the extreme zeros of $f$ in $(-\sqrt{d}, \sqrt{d})$. 
 
If $\deg(f)=2$, then $f=\alpha (x-c_1)(x-c_2)$ with $\alpha\in F_{>0}$ and $c_1+c_2,\,c_1c_2\in F$. By (\ref{CM2}), $f\in \PO(x^2-d)$. 
   Assume $\deg(f)\geq 4$. 

 If $f(\pm \sqrt{d})=0$, then $f=h(x^2-d)$ for some $h\in F[x]$,  $\deg(h)= n-2$ and $h\in \Pos(\CS)$. By induction on degree, $h\in \PO(x^2-d)$, hence $f\in \PO(x^2-d)$. 
    
 Assume $f(\pm \sqrt d)\neq 0$.  Choose $r,s \in F$ such that $-\sqrt{d}<r <c_1<c_2<s< \sqrt{d}$. Then $f\in \Pos((-\infty,r] \cup [s,\infty))$. By (\ref{1 thm}), $f\in \PO((x-r)(x-s))$ and by (\ref{CM2}), $(x-r)(x-s)\in\PO(x^2-d)$. Hence $f\in \PO(x^2-d)$. 
\medskip

\fbox{{\it Case 2.}} Let  $\CS=[-\sqrt d,\sqrt d]=\CS(d-x^2)$ with $d\in F_{>0}$ and $\sqrt d\not\in F$. 
Let  $a\in F$ with $a<-\sqrt d$. We will show $x-a\in \PO(d-x^2)$. 
For $r\in F_{<0}$, 
$$\frac 1{-2r}[(d-x^2)+(-x+r)^2]=x+\frac {r^2+d}{-2r} \in \PO(d-x^2).$$ 
Since $\frac{r^2+d}{2r}=\beta$ has solutions $r=\beta \pm \sqrt{\beta^2-d}$, we can choose $\beta\in (-\sqrt d-\epsilon,-\sqrt d)$ for some $\epsilon> 0 $ such that $r\in F_{<0}$ and
$\frac{r^2+d}{2r}\in (a,-\sqrt d) \cap F$. 
Now 
$$x-a=\left(x-{\frac{r^2+d}{2r}}\right)+\left({\frac{r^2+d}{2r}}-a\right)\in  \PO(d-x^2),$$
as ${\frac{r^2+d}{2r}}-a>0$. 
Similarly, for $a\in F$ with $a>\sqrt d$,  
one can show $-x+a\in \PO(d-x^2)${,} by considering the following equation for $r\in F_{>0}${,}
 $$\frac 1{2r}[(d-x^2)+(x-r)^2]=-x+\frac {r^2+d}{2r} \in \PO(d-x^2).$$ 
Let $f\in \Pos(\CS(d-x^2))$. 

\begin{enumerate}

\item If $f\geq 0$ on $R$, then $f\in \PO(1)$, by (\ref{sos}).

\item If $f(\pm \sqrt d)=0$, then $f=(d-x^2)h$ with $h\in \Pos(\CS)$. {By induction} on degree, $h\in  \PO(d-x^2)$, hence $f\in \PO(d-x^2)$. 

\item If $f(a)=0$ for $a\in \CS^\circ$, then $f'(a)=0$. If $g$ is {the} minimal polynomial of $a$ over $F$, then $f=g^2h$ with $h\in \Pos(\CS)$. {By induction} on degree{,} $h\in \PO(d-x^2)$, hence $f\in \PO(d-x^2)$. 

\item Assume $f>0$ on $\CS$ and $f$ takes a negative value at some point in $(-\infty,-\sqrt d)$. If $a<-\sqrt d$ is a zero of $f$ in $(-\infty,-\sqrt d)$, then we can find $ r\in F$, $r<a$ and $s\in F_{>0}$
 such that the linear polynomial $g=sx+s'\in F[x]$ satisfies $g < f$ on $\CS$ and $g(r)=f(r)$. Now $f-g= (x-r)h$ with $h\in \Pos(\CS)$. {By induction on degree, $h\in \PO(d-x^2)$} and $x-r\in \PO(d-x^2)${,} as shown above, we get $f\in \PO(d-x^2)$. 

\item If $f$ takes a negative value at some point in $(\sqrt d,\infty)$, then replacing $x$ by $-x$, we are in case (4).
\end{enumerate} 

\fbox{{\it Case 3.}} If $\CS$ is not connected, then $\CS$ is of the form $\cup_{i=1}^{r}[a_i,b_i]$ and the following holds.
\begin{enumerate}
    \item $b_i < a_{i+1}$ for all $i$. Further, if $a_1=-\infty$ or $b_r=\infty$ or both, then
we have $(a_1,b_1]$ or $[a_r,b_r)$ or both {respectively}, in the above decomposition of $\CS$.
    \item By our assumption, $[F(b_i):F]=2$ if and only if $[F(a_{{i+1}}):F]=2$ and in this case, $a_{{i+1}}$
and $b_i$ are conjugates over $F$.
\end{enumerate}
Using {(case $1$, case $2$ and (\ref{CM2})) and following the proof
of (\ref{1 thm}),} we can complete the proof in case $3$. 
$\hfill \square$
\end{proof}

\smallskip 
As a consequence of (\ref{t4.5}), we {will write some examples}.

\begin{example} In $\BQ[x]$,
\begin{enumerate}
\item $\Pos([-\sqrt 2,\sqrt 2])=\PO(2-x^2)$.

\item If $$\CS=[-6,-4-\sqrt 2]\cup[-4+\sqrt 2,-\sqrt 2]\cup [\sqrt 2,2] \cup \{3\}\cup [4,\infty),$$ then 
$\Pos(\CS)=\PO(\ul g)$, where 
$$\ul g=\{x+6,(x+4)^2-2,x^2-2,(x-2)(x-3),(x-3)(x-4)\}.$$ $\hfill \square$
\end{enumerate}
\end{example}

\begin{lemma}\label{t4.31}
Let $\CS= \CS(x,d-x^2)=[0,\sqrt d]\subset R$. Then $\Pos(\CS)=\PO(x,d-x^2)$ in $F[x]$.
\end{lemma}

\begin{proof} Follow the proof of (\ref{t4.5}) case 2.
For $r>\sqrt d$, $r\in F$, we get $x-r\in \PO(d-x^2)$ and $\Pos(\CS)=\PO(x,d-x^2)$.
$\hfill\square$
\end{proof}

\begin{lemma}\label{t4.3}
Let $d\in F_{>0}$ and $\CS= \CS(x^2-d,x)=[\sqrt d,\infty)$ with $\sqrt d\not\in F$. Then $\Pos(\CS)$ is not a finitely generated preordering of $F[x]$. 
\end{lemma}

\begin{proof}
      Let $f_1,\ldots,f_{{m}}\in \Pos(\CS)$. {Assume $\Pos(\CS)=\Pos(f_1,\ldots,f_m)$, where $f_1,\ldots,f_s$ are linear and $\deg f_j>1$ for $j>s$}. Let $x-r\in F[x]\cap \Pos(\CS)$ with $r<\sqrt d$.
    If 
    $$x-r=c_1f_1+\ldots+c_sf_s$$ with $c_i\in F_{>0}$, then all $f_i$ are linear (note that every $f\in \Pos(\CS)$ has {a} positive leading coefficient). Now choosing $r$ {such} that $r<r_i$, where $r_i$ {is the root} of $f_i$, we get $f_i(r){<}0$ {for all $i=1,\ldots,s$}. This contradicts that $0=\sum_1^s c_if_i(r)$. Thus $\Pos(\CS)$ is not a finitely generated preordering of $F[x]$.
    $\hfill \square$
\end {proof}

\begin{lemma}
Let $d\in F_{>0}$ and  $\CS= \CS(d-x^{2^n})=[-\sqrt[2^{n}] d,\sqrt[2^{n}]d\;]$ with  $\sqrt[2^n] d\not\in F$. Then  $(\alpha \pm x) \in \PO(d-x^{2^n})$ for all $\alpha \in [\sqrt[2^n]d, \infty) \cap F$.
 \end{lemma}
 
\begin{proof}
 For $r \in F_{>0}$,  $$\frac{1}{2r} \left[ d - x^{2^n} + (x^{2^{n-1}} \pm r)^2 \right] 
  = \pm x^{2^{n-1}} + \frac{r^2 + d}{2r} \in \PO(d-x^{2^n}).$$  
 On $R^+$,  the minimum value of $a_1= \frac{r^2 + d}{2r}$ is $\sqrt d$. Hence $(a_1 \pm x^{2^{n-1}}) \in \PO(d - x^{2^n})$ for all $a_1\in [\sqrt d, \infty) \cap F$. 

Repeating the same calculation with $a_1-x^{2^{n-1}}$, we get $a_2\pm x^{2^{n-2}} \in \PO(d-x^{2^n})$ for all $a_2\in [\sqrt[4]d,\infty)\cap F$. Continuing this process, we get $a_n\pm x\in \PO(d-x^{2^n})$ for all $a_n\in [\sqrt[2^n]d,\infty)\cap F$. $\hfill \square$ 
 \end{proof}

\begin{remark}\label{5.9} 
   \begin{enumerate}
    \item  Let $\ul g=\{x,2-x^3\} \subset F[x]$ and $\CS =\CS(\ul g)=[0, \sqrt[3]2] \subset R$. We show ${a - x} \in \PO(\ul g)$ for all $a \in [\sqrt 2, \infty) \cap F$.
Note that 
$$x(2 - x^3) + (x^2 - 1)^2 + (x - 1)^2 = 2 - x^2 \in \PO(\ul g).$$ For $r \in F_{>0}$, 
$$\frac{1}{2r} \left[ 2 - x^2 + (x - r)^2 \right] 
    = -x + \frac{r^2 + 2}{2r} \in \PO(\ul g).$$
On $R^+$,  the minimum value of $a=\frac{r^2 + 2}{2r}$ is $\sqrt 2$. Hence $(a - x) \in \PO(\ul g)$ for all $a \in [\sqrt 2, \infty) \cap F$. 

    \item Let $\ul g=\{x,x^2-2,3-x^2\}$ and $\CS = \CS(\ul g)=[\sqrt 2,\sqrt 3]$. 
 If $\alpha\in (-\infty, 0] \cap F$, then $(x-\alpha)\in \PO(\ul g)$. For $r \in F_{>0}$, 
 $$\frac{1}{2r} \left[ 3 - x^2 + (x - r)^2 \right] 
 = -x + \frac{r^2 + 3}{2r} \in \PO(\ul g).$$ 
 On $R^+$,  the minimum value of $a=\frac{r^2 + 3}{2r}$ is $\sqrt{3}$. Hence $(a - x) \in \PO(\ul g)$ for $a \in [\sqrt{3}, \infty) \cap F$.

    \end{enumerate}
\end{remark}

In view of (\ref{t4.1}, \ref{t4.5} and \ref{5.9}), we can ask the following.

\begin{question}\label{Qn} 
Is $\Pos(\CS)$  a finitely generated preordering of $F[x]$ in the following closed and bounded cases with $\partial \CS \not\subset F$?
\begin{enumerate} 

\item Let $\CS= \CS(x,2-x^3)=[0, \sqrt[3]2]$. Does $(a - x) \in \PO(x, 2 - x^3)$ for all $a \in [\sqrt[3]2, \sqrt 2] \cap F$?

\item 
 Let $\CS= \CS(x,x^2-2,3-x^2)=[\sqrt 2,\sqrt 3]$. Does $(x-\alpha) \in \PO(x, x^2 -2 ,3 - x^2)$ for all $\alpha \in (0, \sqrt{2}] \cap F$?

\end{enumerate}
\end{question}


\section{A question of Powers} \label{8.0}

In this section, we give a partial answer to a question of Powers \cite{Vict2}.  For the definition of real spectrum $\text{Sper}(A,S)$ and semi-ordering, see \cite[Section 2.5, 5.3]{Mur1}.

The following result is from \cite[Ex 5.5.2]{Prest}.

\begin{lemma}\label{order}
Let $\BF\subset \mathbb R$ be an algebraic extension of $\mathbb Q$. Then every semi-ordering on $\BF$ is an ordering  {on $\BF$}. 
\end{lemma}

The following result is from  \cite[Thm 5.3.2]{Mur1}.

\begin{theorem}\label{Semi-order}
    Let {$\mbox{M} $} be a quadratic module of $A=\mathbb Q[x_1,\ldots,x_n]$  and $a \in A$. Then  $a\in {\mbox{Q}\setminus -\mbox{Q}}$ for all semi-orderings {$\mbox{Q}$} of $A$ containing {$\mbox{M} $}  if and only if $-1 \in{\mbox{M}}  - \sum A^2a$ (i.e $-1=q-pa$ for some $q\in {\mbox{M}}$ and $p\in \sum A^2$).
\end{theorem} 

{
The following result answers Powers question \cite[Remark 8]{Vict2} in a special case.
}

\begin{theorem}\label{quad in rp}
Let $\ul g=\{g_1, \ldots, g_s\}\subset A=\mathbb Q[x_1,\ldots,x_n]$ such that
$\QM(\ul g)_\mathbb{R}$ is an archimedean quadratic module of $A_{\mathbb R}$ and $\dim (A/I)=0$, where {$I=\QM(\ul g)\cap -\QM(\ul g)$}. Then $\QM(\ul g)$ is an archimedean quadratic module of $A$.
\end{theorem}

\begin{proof}
    Since the quadratic module $\QM(\ul g)_{\mathbb R}$ of $A_{\mathbb R}$ is  archimedean, we get $\CS(\ul g)\subset\BR^n$ is compact. Choose $N\in \mathbb N$ {such that
$$a=  N-\sum_{i=1}^ n x_i ^2 > 1 ~~\mbox{  on  } ~~
\CS(\ul g)$$ }
(so $a-1>0$ on $\CS(\ul g))$.  By (\ref{Pow thm1}), $a-1\in \PO(\ul g)$, so $a\in 1+\PO(\ul g)$. Therefore, $a>0 $ on    $\text{Sper}(A,\QM(\ul g))$, i.e. 
$$a\in Q\setminus -Q \hspace*{.1in} \mbox{for all orderings}\hspace*{.1in} Q \hspace*{.1in}\mbox{of}\hspace*{.1in} A \hspace*{.1in}\mbox{containing}\hspace*{.1in} \QM(\ul g). \hspace*{1in} (\star)$$

For any semi-ordering $\mbox{Q}$ of $A$ containing $\QM(\ul g)$, $\mbox{Q}\cap -\mbox{Q}$ is a prime ideal of $A$ containing {$I$}. Since {$\dim(A/I)=0$}, we get $\mbox{Q}\cap -\mbox{Q}$  is a maximal ideal of $A$. Therefore,  $L= A/(\mbox{Q}\cap -\mbox{Q})$  is an algebraic extension of $\mathbb Q$. By (\ref{order}), the extension  of the semi-ordering $\mbox{Q}$ is an ordering $\ol {\mbox{Q}}$ in $A/(\mbox{Q}\cap -\mbox{Q})$. Since the contraction of an ordering is an ordering \cite[Sec 2.4]{Mur1} and $\mbox{Q}$  is the contraction of $\overline{\mbox{Q}}$, we get that $\mbox{Q}$  is an ordering of $A$.

      By $(\star)$ and (\ref{Semi-order}),  $-1\in \QM(\ul g)-\sum A^2 a$. Thus,  
      $$1+q=pa$$ for some
 $q\in \QM(\ul g)$ and  $p\in \sum A^{2}$.
 So $$
(1+q)a = pa^2 \in \sum A^2.$$ Rest of the proof is similar to W\"{o}rmann's \cite[Thm 6.1.1]{Mur1}. We will complete the proof below.

Since the quadratic module $\QM(a)$ of $A$ is archimedean, there exists  $m \in \BN$ such that $m -q \in \QM(a)$. Then $m -q = t_1 + at_2$ for some $t_1, t_2 \in \sum A^{2}$. So, using $1+q=pa$, we get
\[
(m -q)(1 + q) = t_1(1 + q) + at_2(pa) =(t_1+t_2pa^2)+t_1q\in \QM(\ul g)
\]
and
\[
m + \frac{m^2}{4} - q = (m - q)(1 + q) + \left(\frac{m}{2} - q\right)^2 \in \QM(\ul g).
\]
Thus, noting that $q\sum x_i^2\in \QM(\ul g)$, we get that
\[N\left(m + \frac{m^2}{4} - q\right)+(1+q)a+q\sum x_i^2 =
N\left(\frac{m}{2} + 1\right)^2 - \sum_{i=1}^n x_i^2 \in \QM(\ul g).
\]
Hence, $\QM(\ul g)$ is an archimedean quadratic module of $A$. $\hfill \square$
\end{proof}

\appendix

\section{Applications of (\ref{1 thm}) :  {$\Pos(K)$ for $K\subset \BR$   closed non-semi-algebraic set}}

 In this section, we will study { $\Pos(K)$} for some closed non-semi-algebraic subsets { $K$} of $\BR$. 
As an application of (\ref{1 thm}), { we}  describe $\Pos({ C})$, where 
${ C}=\cap_{n\geq 0} C_n$ is the Cantor set. Here $C_{0}=[0,1]$, $C_1=[0,\frac 13]\cup [\frac 23,1]$, $C_2=[1,\frac 19]\cup [\frac 29, \frac 39]\cup [\frac 69, \frac 79]\cup[\frac 89, 1]$, etc. 

\begin{theorem}\label{t.2}
Let $\mathcal N_n\subset \BF[x]$  the natural choice of generators of the closed semi-algebraic set $C_n$. Then
   $$\Pos({ C})= \bigcup_{n\geq 0} \Pos(C_n)=\bigcup_{n\geq 0} \PO( \CN_n).$$ In particular, $\Pos({ C})$ is not a finitely generated preordering of $\mathbb{F}[x]$.
\end{theorem}

\begin{proof} 
We only need to show that if $f\in  \Pos({ C})$,  then $f\in \Pos(C_n)$ for some $n\in \mathbb{N}$. Suppose $f(x)<0$ for some $x\in [0,1]$. Let $$\alpha=\min\{c'-c: f(c)=f(c')=0~\text{and} ~f<0~\text{on}~(c,c') \}.$$ 
Each connected component of $C_n$ has length $1/3^n$. Choose $n\in \mathbb{N}$  such that $1/3^n < \alpha$. Note  $\partial C_n \subset { C}$, hence $f(\partial C_n)\geq 0$. If $J_n$ is a connected component of $C_n$ and $f(x)<0$ for some $x\in J_n\setminus \partial C_n$, then there exists $c,c'\in J_n$ such that $f(c)=f(c')=0$. Here $|c'-c|\leq 1/3^n <\alpha$, which is a contradiction. So $f\in \Pos(C_n)$. 
 $\hfill \square$
\end{proof}

\begin{remark} 
 If $\Pos^{\text{cont}}({ C})$ denotes the set of all continuous functions on $[0,1]${ ,} which are non-negative on the Cantor set ${ C}$, then 
 $$\Pos^{\text{cont}}({ C})\neq \bigcup_{n\geq 0} \Pos^{\text{cont}}(C_n).$$ Thus  (\ref{t.2}) is not true { when} $\BF[x]$ is replaced by the ring of continuous functions from $[0,1]$ to $\BR$. To see this, we can define a continuous piecewise linear function $f:[0,1]\to [-1,1]$ such that $|f(x)|\leq x$ and consisting of infinitely many line segments $l_n$ defined as follows: $l_1$ joins  $(1,1)$ to $(\frac 23,0)$ until it hits $y=-x$ line at $P_1$, $l_2$ joins $P_1$ and $(\frac 13,0)$ until it hits $y=x$ line at $P_2$, $l_3$ joins $P_2$ and $(\frac 2 {3^2},0)$ until it hits $y=-x$ line at $P_3$, $l_4$ joins $P_3$ and $(\frac 1 {3^2},0)$ until it hits $y=x$ line at $P_4$, $l_5$ joins $P_4$ and $(\frac 2 {3^3},0)$ until it hits $y=x$ line at $P_5$, etc. Clearly, $f\in \Pos^{\text{cont}}({ C})\setminus\bigcup_{n\geq 0} \Pos^{\text{cont}}(C_n)$.   
 $\hfill \square$
\end{remark}

 \begin{theorem}\label{t.1}
 Let  ${ K} \subset \BR$ be a closed set of the following types, where $\alpha_i,\beta_i \in \BF$ with $\alpha_i \leq \beta_i < \alpha_{i+1}$ and $\CN_n$ is the natural choice of generators of ${ K_n}$. Then $ \displaystyle \Pos({ K})=\displaystyle\bigcup_{n\geq 0}\Pos({ K_n}) =\bigcup_{n\geq 0} \PO(\CN_n)$ in the following cases.
 \begin{enumerate}

\item\label{e.1}  
${ K}=\displaystyle\bigcup_{i\geq 0}[\alpha_i,\beta_i]$ 
and ${ K_n}=\displaystyle \bigcup_{i=0}^n [\alpha_i,\beta_i] \cup [\alpha_{n+1},\infty)$
with $\displaystyle \lim_{i\to \infty} \alpha_i = \lim_{i\to \infty}\beta_i=\infty$. Further,
$[\alpha_0,\beta_0]=(-\infty,\beta_0]$ is allowed. 

 \item 
 ${ K}=\displaystyle\bigcup_{i\leq 0}[\alpha_i,\beta_i]$  and ${ K_n}=\displaystyle\bigcup_{i=-n}^0 [\alpha_i,\beta_i] \cup (-\infty,\beta_{-(n+1)}]$ with  $\displaystyle\lim_{i\to -\infty} \alpha_i = \lim_{i\to -\infty}\beta_i=-\infty$. Further,
$[\alpha_0,\beta_0]=[\alpha_0,\infty)$ is allowed.

\item  ${ K}=\displaystyle\bigcup_{i\in \mathbb{Z}} [\alpha_i,\beta_i]$  and ${ K_n}=\displaystyle\bigcup_{i=-n}^n [\alpha_i,\beta_i]\, \cup \,(-\infty,\beta_{-(n+1)}] \,\cup\, [a_{n+1},\infty)$ with $\displaystyle\lim_{i\to \pm \infty} \alpha_i = \lim_{i\to \pm \infty}\beta_i= \pm \infty$. 

\item ${ K}=\{0\}\displaystyle\bigcup_{i\geq 1} \left \{\frac 1 i \right\}$ and ${ K_n}=\left [0,\frac 1{n+1}\right ]\displaystyle\bigcup_1^n \left \{\frac 1i \right \}$. 

\end{enumerate}

Note that above closed sets ${ K}$   are not semi-algebraic and $\Pos({ K})$ are not finitely generated preorderings of $\BF[x]$.
\end{theorem}

\begin{proof}
    We only need to show that if $f\in \Pos({ K})$, then $f\in \Pos({ K_n})$ for some $n$. This follows from the fact that $f$ can oscillate only finitely many times, so eventually it will be strictly monotone on ${ K}$. Thus $f\in \Pos({ K_n})$ for some $n$. Use (\ref{1 thm}) for $\Pos({ K_n})= \PO(\mathcal N_n)$. $\hfill \square$
\end{proof}
\medskip

\begin{example} Let us write  
\begin{itemize} 
\item [$(n_1)$] $I_i=[2i,2i+1]$, $t_i=2i+1+\frac 12$ and $I_i'=I_i\cup \{t_i\}$.

\item [$(n_2)$] $f_i=(x-2i-1)(x-2i-2)$, $g_i=  (x-2i-1)(x-t_i)$ and $h_i=(x-t_i)(x-2i-2)$.

\item [$(n_3)$] $J_n=\displaystyle \bigcup_{i=0}^n I_i\cup [2n+2,\infty).$

{ \item [$(n_4)$] $ J_n'= \displaystyle\bigcup_{i=-n}^0 I_i \cup (-\infty,-2n-1]$.}

\item [$(n_5)$] $J_n''= \displaystyle\bigcup_{i=0}^n (I_i\cup \{t_i\})\cup 
       [2n+2,\infty).$ 
       
{ \item [$(n_6)$] $J_n'''= \displaystyle\bigcup_{i=0}^n (I_i\cup I_{-i})\cup (-\infty,-2n-1]\cup [2n+2,\infty)$.}
\end{itemize}

Then  (\ref{t.1}) gives the following examples.

\begin{eqnarray*}
  \Pos\left(\bigcup_{i\geq 0} I_i\right) &=&  \bigcup_{n\geq 0} \Pos(J_n)
 =  \PO(x,f_0,f_2, \ldots ). \\
\Pos\left(\bigcup_{i\leq 0} I_i\right) &= & 
\bigcup_{n\leq 0} \Pos(J_n') = \bigcup_{n\geq 0} \PO(1-x,f_{-1},\ldots,f_{-(n+1)})
= \PO( 1-x,f_{-1}, f_{-2}, \ldots).
\\
     \Pos\left(\bigcup_{i\geq 0} I_i'\right)
     & =&
      \bigcup_{n\geq 0} \Pos(J_n'') 
      = \bigcup_{n\geq 0} \PO(x,g_0,\ldots, g_n,h_0,\ldots,h_{n})= \PO(x,g_0, h_0, g_1, h_1,\ldots).
  \\
     \Pos\left(\bigcup_{i\in \mathbb Z} I_i\right)  &=&
     \bigcup_{n\geq 0} \Pos(J_n''') 
 =  \bigcup_{n\geq 0}  \PO(f_{-n},\ldots,f_0,\ldots, f_n) = \PO({f_0,f_1, f_{-1}, f_{2}, f_{-2},\ldots}).
\end{eqnarray*}
$\hfill \square$
\end{example} 

\section{Powers question revisited (added during galley proof)}

In the proof of Thm \ref{quad in rp}, the assumption, that $\QM(\ul g)_{\mathbb R}$  is archimedean, is not really used. We 
only used that $\dim(A/I)=0$. We can generalize Thm \ref{quad in rp} as follows.

\begin{theorem}
    Let $\ul g\subset A=\BF[x_1,\ldots,x_n]$ and $I=\QM(\ul g) \cap -\QM(\ul g)$  be the support of the 
quadratic module $\QM(\ul g)$. Assume that $V_{\mathbb R}(I)=\{a\in \BR^n : f(a)=0 ~\forall f \in I \}$ is a compact set 
in $\BR^n$. Then $\QM(\ul g)$ is an archimedean quadratic module of $A$. As a consequence of Thm \ref{Pow thm2},
if $f \in A$ is $>0$ on $\CS(\ul g)$, then $f\in \QM(\ul g)$.
\end{theorem}

\begin{proof}
    Let $h_1,\ldots,h_r\in A$ be a set of generators of the ideal $I$. Then, for any $p\in A$, the identity
    $$ ph_i=h_i (\frac{p+1}2)^2+ (-h_i)(\frac {p-1}2)^2$$
shows that $$I\subset \PO(\pm h_1,\ldots,\pm h_r) \subset \QM(\ul g).$$
Since $\CS(\pm h_1,\ldots,\pm h_r)=V_{\mathbb R}(I)$ is compact, by Thm \ref{Pow thm2}, there exist $N\in \BN$ such that $N-\sum_1^n x_i^2 \in \PO(\pm h_1,\ldots,\pm h_r)$. Thus $N-\sum_1^n x_i^2 \in \QM(\ul g)$. Therefore, $\QM(\ul g)$ is an archimedean quadratic module of $A$. \qed
\end{proof}

As a consequence, the Powers question reduces to the following

\begin{question}
    Let $\ul g\subset \BF[x_1,\ldots,x_n]$ and $I=\QM(\ul g) \cap -\QM(\ul g)$  be the support of the 
quadratic module $\QM(\ul g)$ of $\BF[x_1,\ldots,x_n]$. Assume that the algebraic set $V_{\mathbb R}(I) \subset \BR^n$ is  non-compact 
and $\BQ(\ul g)_{\mathbb R}$ is an archimedean quadratic module of $\BR[x_1,\ldots,x_n]$. Is $\QM(\ul g)$ is an archimedean quadratic module of $\BF[x_1,\ldots,x_n]$?
\end{question}

\medskip

\noindent\textbf{Acknowledgement:} 
We would like to thank the referees for their detailed and thoughtful comments and suggestions and for references $[2, 17, 19]$. It has improved the readability of the paper substantially. On referees suggestion, we have worked out our main result (\ref{1 thm}) for $F$ (earlier version was written for $\BF$). Thanks to Charu Goel and J.K. Verma who had asked similar question to us.
Finally, the second and third authors would like to thank CSIR  and NBHM  resp. for their PhD fellowships.

\bibliographystyle{amsplain}

\Addresses

\end{document}